\documentclass[twoside,11pt]{amsart}
\usepackage{latexsym,amsmath,amsfonts,amssymb,soul,cite,lineno,cancel}
\usepackage{hyperref,enumitem}
\usepackage{,amsthm,verbatim,mathtools}
\newtheorem{thm}{Theorem}[section]
\newtheorem{deff}[thm]{Definition}
\newtheorem{lem}[thm]{Lemma}

\newtheorem{rem}[thm]{Remark}
\newtheorem{prop}[thm]{Proposition}
\newtheorem{cor}[thm]{Corollary}

\newtheorem {ex}[thm]{Example}

\newcommand{\vG}{\varGamma}
\newcommand{\ve}{\varepsilon}

\newcommand{\vO}{\varOmega}
\newcommand{\vS}{\varSigma}

\newcommand{\ov}{\overline}

\def\N{{\mathbb N}}
\def\R{{\mathbb R}}

\def\mcL{{\mathcal L}}
\def\mcS{{\mathcal S}}

\def\mcR{\mathcal R}

\def\emp{\emptyset}

\def\lg{\langle}
\def\rg{\rangle}
\def\mcN{\mathcal N}

\def\be{\begin{equation}}
 \newcommand{\dint}{\displaystyle{\int}}
\def\ee{\end{equation}}
\def\ba*{\begin{eqnarray*}}
	\def\ea*{\end{eqnarray*}}
\def\ba{\begin{eqnarray}}
\def\ea{\end{eqnarray}}
\def\smi{\setminus}

\begin{document}
\title[Representations of multimeasures via  $BDS_m$-integral]{ Representations of multimeasures via the  multivalued Bartle-Dunford-Schwartz integral}
\author[L. Di Piazza, K. Musia{\l}, A. R. Sambucini]{Luisa Di Piazza, Kazimierz Musia{\l}, Anna Rita Sambucini}
\newcommand{\Addresses}{{
  \bigskip
  \footnotesize
\textit{Luisa Di Piazza}:
 Department of Mathematics, University of Palermo, Via Archirafi 34, 90123 Palermo, (Italy).
 Email: luisa.dipiazza@unipa.it, Orcid ID: 0000-0002-9283-5157\\
\textit{Kazimierz Musia{\l}}:
  Institut of Mathematics, Wroc{\l}aw University, Pl. Grunwaldzki  2/4, 50-384 Wroc{\l}aw, (Poland).
  Email: kazimierz.musial@math.uni.wroc.pl, Orcid ID: 0000-0002-6443-2043 \\
\textit{Anna Rita Sambucini\thanks{ (corresponding author)}}:
 Department of Mathematics and Computer Sciences, 06123 Perugia, (Italy).
 Email: anna.sambucini@unipg.it, Orcid ID: 0000-0003-0161-8729; ResearcherID: B-6116-2015.
}}

\begin{abstract}
An integral for a scalar function with respect to a multimeasure $N$ taking its values in a locally convex space is  introduced.
The  definition is independent of  the selections of $N$  and is
related to a functional version of the Bartle-Dunford-Schwartz integral with respect to a vector measure presented by Lewis.
Its properties  are studied together with its application to Radon-Nikod\'{y}m
 theorems in order to represent  as an integrable derivative the ratio of  two general multimeasures or two $d_H$-multimeasures;
 equivalent conditions are provided in both cases.
\end{abstract}
\maketitle

\section{Introduction}\label{s-intro}
The theory of the Bartle-Dunford-Schwartz-integral (BDS-integral) of a scalar function with respect to a vector  valued measure was introduced
in 1955 by R. C. Bartle, N. Dunford and J. T. Schwartz \cite{bds}, and subsequently extensively
studied by several authors (\cite{Lew,mu,MR2419122,MR3098470,MR3473006,kk,k2,wzw}).

 In   `70 Lewis \cite{Lew}
 proposed an equivalent functional version of the integral  (see also \cite{k2,kk,MR2419122}). The literature concerning the Bartle-Dunford-Schwartz
 integration is rather wide  so we quote here only those papers which are  close to the topic of our work and two books (\cite{kk,MR2419122}).\\
 In \cite{mu} the second author proved the existence of the Radon-Nikod\'{y}m
derivative of a vector valued measure $\nu$ with respect to a vector valued measure $\kappa$
by means of the BDS-integral,  under suitable hypotheses on the measures $\nu$ and $\kappa$.
 From this article then came out  \cite{cala},
 where the derivative belongs to a suitable space
and that of \cite{anca}, where a
Radon-Nikod\'{y}m theorem was given  bases on a construction of Maynard type.

In \cite{ka} Kandilakis defined an integral of a scalar function with respect to a multimeasure $N$, whose values were weakly compact and convex
subsets of a Banach space.   This integral is constructed using the selections of $N$.
Contrary to the classical BDS-integral, the Kandilakis' integral  is only sublinear with respect to the integrable functions.

 It is our aim to  define an  integral with respect to a multimeasure $N$,
independent of  the selections of $N$. To achieve it   we take into account the support functions of $N$.
  Our approach is a consequence of Kandilakis' calculations. The considered multimeasures $N$ are very general, in fact they can take
   nonempty closed, convex  values in an arbitrary  locally convex space $X$.
Sometimes we will assume quasi-completeness of $X$ (closed bounded sets are complete).

The theory of  multifunctions and multimeasures is an interesting field of research since it has applications in various applied sciences. In particular, recently, interval-valued multifunctions and multimeasures have been applied also to signal and image processing, see for example \cite{ccgis} and the references therein.

 In Section \ref{s-prel} we study  properties of the integral. Moreover, we compare the Aumann-BDS integral and the
new one: if the multimeasure $N$ possesses selections and the scalar integrable function is bounded, then the Aumann-BDS integrability
implies the $N$-integrability.  To obtain the main results we require the existence of control measures for the multimeasures under investigation and in
Section \ref{control} we give conditions guaranteeing  the existence of such controls.
In particular in Theorem \ref{T1} a characterization of its existence through the countable chain condition (ccc) is provided, while in  Theorem \ref{t5} we describe the class of Banach spaces with all $cb(X)$-multimeasures being $d_H$-multimeasures (and so admitting control measures).
In Theorem \ref{esempi}   a class of spaces  is given in which 
every $c(X)$-valued multimeasures has (ccc).  We quote also very recent characterization of Rodriguez \cite{Ro} of Banach spaces with all multimeasures admitting  control measures. 
  As an application of the new integral,  we study the problem of existence of the Radon-Nikod\'{y}m derivative of a given  multimeasure $M$  with
respect to a multimeasure $N$.
In the classical measure theory, the Radon-Nikod\'{y}m theorem states in concise conditions, namely absolute continuity, domination and
subordination, how a measure can be factorized by another  measure through a density function.
In Section \ref{s-arbitrary} we first consider   arbitrary multimeasures.
In such a case  we are able to characterize  the existence of  the
Radon-Nikod\'{y}m derivative  of $M$ with respect to $N$ (see Theorem
\ref{t1})  by means of the notions of  uniform scalar
absolutely continuity,  uniform scalar domination and uniform scalar subordination.
  In Section \ref{s-dH}   multimeasures which are countably additive in the Hausdorff metric are studied.
   In such a case the multimeasures take values in  nonempty, bounded, closed, convex  subsets of a locally convex space.
The differentiation of $M$ with respect to $N$ is in general not
equivalent to differentiation
  of its R{\aa}dstr\"{o}m embedding $j\circ{M}$ with respect to $j\circ{N}$.
  The reason is that one has to take into account also integration with
respect to the   measure $j\circ(-N)$.
In the case of  multimeasures, we have in general
$j\circ(-N)\neq-j\circ{N}$;  moreover, since  the proofs given are based on support functions which obviously lost the additivity, the methods of our proofs  are not a mere repetition of the vector case.
   The main result is Theorem \ref{t4} where we find conditions (the
strong  uniform scalar absolutely continuity,  the strong uniform
scalar domination and the strong uniform scalar subordination)  guaranteing the existence of  the Radon-Nikod\'{y}m
derivative both of $M$ with respect to $N$  and of their
  R{\aa}dstr\"{o}m embeddings.
  At last in Section \ref{s-ex} we provide some examples of multimeasures that can be represented by another multimeasure through a
   multivalued Bartle-Dunford-Schwartz density.

\section{Preliminaries}\label{s-prel}
Throughout $(\vO,\vS)$ is a  measurable space,
the real numbers are denoted by $\mathbb R$ and $\mathbb R_0^+$ denotes the non-negative reals. If $\nu:\vS\to(-\infty,+\infty]$ is a measure,
 then $|\nu|$ denotes its variation.  $\vS_E $ is the  family of all
 $\vS$-measurable subsets of $E$.  Let $X$ be a locally convex linear topological space (shortly, locally convex space) and let $X'$ be its conjugate space.
Given a subset $S$ of $X$, we write {\rm co}$(S)$, {\rm aco}$(S)$ and {\rm span}$(S)$
to denote, respectively, the convex, absolutely convex and linear hull of $S$.\\
The symbol $c(X)$ denotes the collection of all
nonempty closed convex subsets of $X$ and $cb(X),\,cwk(X)$,  $ck(X)$   denote respectively
the family of all bounded and  the family of all (weakly) compact  members of $c(X)$.  For every $C \in c(X)$ the
\textit{  support function of}   $\, C$ is denoted by $s( \cdot, C)$ and
defined on $X'$ by $s(x', C) = \sup \{ \langle x',x \rangle \colon  \ x \in C\}$, for each $x' \in X'$.
The symbol $\oplus$ denotes the closure of the Minkowski addition.\\
We say that $M:\vS\to{c(X)}$ is a \textit{ multimeasure} if for every   $x'\in{X'}$ the set function
$s(x',M(\cdot)):\vS\to (-\infty,+\infty]$ is a  $\sigma$-finite measure.
 The $\sigma$-finiteness of each $s(x',M)$ seems to be the weakest possible assumption,  otherwise we meet problems connected with the RN-theorem.
  Simply if some $\nu:=s(x',M)$ is not $\sigma$-finite but is absolutely continuous with respect to a finite measure $\mu$, then there is a set $\vO_0$
  such that $\nu$ is $\sigma$-finite on $\vO_0$ and takes only values $0,+\infty$ on $\vO^c$. If $\nu(E)=\int_E f\,d\mu$, then $f$ must take infinite
 values and we are not interested in such a situation. \\
   Given a multimeasure $M : \vS \to c(X)$ we denote by
${\mathcal N}(M) := \{E\in\vS\colon M(E)=\{0\}\}$ the family of null sets. A multimeasure
  $M:\vS\to{c(X)}$ is said to be $\sigma$-{\it bounded} if there is a sequence $(\vO_n)_n$ of elements of $\vS$ such that $\vO\smi\bigcup_n\vO_n\in\mcN(M)$ and
   $M$ is $cb(X)$-valued on each algebra $\vS_{\vO_n}$.

A multimeasure $M$ is called \textit{ pointless} if its restriction to no set $E\in\vS\smi\mcN(M)$ is a vector measure. Each multimeasure
 determined  by a function  (see \cite{cdpms2019} for the  definition) that is not  scalarly equivalent to zero function  is pointless.
 Less trivial examples can be deduced from \cite[Example 1.11]{mu4} if one assumes that the function $r$ appearing there is strictly positive.
  Two multimeasures $M,N:\vS\to{c(X)}$ are {\it consistent}
   if there exists $H\in\vS$ such that $M$ and $N$ are pointless on $H$ and  vector measures on $H^c$.

 If $M$  is a $cb(X)$-valued multimeasure, then for each  $x'\in{X'}$ the measure $s(x',M)$ is finite.
 If $A\in\vS$, then $M|_A$ is the multimeasure defined on $\vS_A$ by $(M|_A)(E):=M(A\cap{E})$.
A multimeasure $M:\vS\to{c(X)}$ is called \textit{  positive}, if $0\in{M(E)}$ for each $E\in\vS$.\\

If $X$ is a Banach space and a multimeasure $M:\vS\to{cb(X)}$ is countably additive in the Hausdorff metric $d_H$, then it is called a $d_H$-\textit{multimeasure}.

  In the following if  $Z$ is any metric space,  we use the  symbol $B_Z$  to denote its closed unit ball.  \\
A helpful tool to study the $d_H$-multimeasures is the R{\aa}dstr\"{o}m embedding
 $j:cb(X)\to \ell_{\infty}(B_{X'})$, defined
 by $j(A):=s(\cdot, A)$, (see, for example  \cite[Theorem 3.2.9 and Theorem 3.2.4(1)]{Beer} or \cite[Theorem II-19]{CV})
  It is known that $B_{X'}$ can be embedded into $\ell_{\infty}'(B_{X'})$  by the mapping $x'\longrightarrow e_{x'}$, where
$\langle{e_{x'},h}\rangle=h(x')$, for each $h\in\ell_{\infty}(B_{X'})$. Moreover, the range of $B_{X'}$ is a norming subset of
 $\ell_{\infty}'(B_{X'})$.
The embedding
$j$  satisfies the following properties:
\begin{description}
\item[\ref{s-prel}.a)] $j(\alpha A \,  \oplus \,  \beta C) = \alpha j(A) + \beta j(C)$ for every $A,C\in  cb(X),\,\, \alpha, \beta \in  \mathbb{R}{}^+$;
\item[\ref{s-prel}.b)] $d_H(A,C)=\|j(A)-j(C)\|_{\infty},\quad A,C\in  cb(X)$;
\item[\ref{s-prel}.c)] $j(cb(X))$  is a closed cone in the space
$\ell_{\infty}(B_{X'})$ equipped with the norm of the uniform convergence.
\end{description}
Observe that instead of $\ell_{\infty} (B_{X'})$, we may use  $C_B (B_{X'}, \tau_{(X',X)})$ (where $\tau_{(X',X)}$
  is the Mackey topology) for weakly compact sets, $C(B_{X'},\sigma(X',X))$ in case of compact sets  and $C_B(B_{X'}, \|\cdot\|_{\infty})$ in case of closed bounded sets.
  But as we do not apply any special properties of  space in which we embed $cb(X)$, we will stay with
\begin{eqnarray}\label{eq:Y}
  \ell_{\infty} (B_{X'}).
  \end{eqnarray}
We denote by  $R(M_E)$ the range of the multimeasure $M$ in $c(X)$, restricted to measurable subsets of $E$:
$R(M_E):= \{ M(F): F \in \vS_E \} \subset c(X)$. Moreover, $\mathcal{R}(M_E):=
\{ z \in X : \exists\;F \in \vS_E, \,\, z \in M(F)\} \subset X$. \\
We say that  a multimeasure $M:\vS\to{c(X)}$ is  \textit{ absolutely continuous} with respect to a multimeasure $N:\vS\to{c(X)}$ (we write then $M\ll{N}$), if
${\mathcal N}(N) \subset  {\mathcal N}(M).$
 If $M\ll{N}$ and $N\ll{M}$, then the multimeasures are called equivalent. 

  We say that a non negative  measure $\mu$ is a \textit{ control measure} for a  multimeasure $M:\vS\to{c(X)}$    if $\mu$ is a finite measure and for each $E\in \vS$ the
  condition $\mu(E)=0$ yields $M(E)=\{ 0 \}$. If $\mu$ is a control measure for $M:\vS\to{c(X)}$, then $M$ is $\sigma$-bounded if and only if it is locally bounded, i.e.
   for each set $E\notin\mcN(\mu)$ there exists a subset $F$ of $E$ of positive $\mu$-measure such that $M$ is $cb(X)$-valued on $\vS_F$. It is a direct consequence
   of \cite[Theorem 1.4]{bds} that each $d_H$-multimeasure has a control measure.

If $M:\vS\to{c(X)}$ is a multimeasure, then $\mcS_M$ denotes the family of all countably additive $X$-valued selections of $M$.\\
\begin{deff}\label{def21}
\rm \mbox{ }
\begin{itemize}
\item
Let $n:\vS\to{X}$ be a vector measure.    A  measurable function  $f:\vO\to\R$
 is called \textit{ Bartle-Dunford-Schwartz (BDS) integrable with respect to} $n$, if for every
 $x' \in X'$ $f$ is $\lg{x',n}\rg$-integrable  and for each  $E\in\vS$ there exists a point
  $\nu(E)\in{X}$ such that $\lg{x',\nu(E)}\rg=\dint_Ef\,d\lg{x',n}\rg$, for every $x'\in{X'}$.
  Such an approach to the Bartle-Dunford-Schwartz integral  was suggested by Lewis \cite{Lew}.
\item
Let $N:\vS\to {c(X)}$ be a multimeasure.  A  measurable function  $f:\vO\to\R$
	is called \textit{  Aumann-Bartle-Dunford-Schwartz (Aumann-BDS)-integrable with respect to} $N$ if $\mcS_N\neq\emp$ and
$f$ is integrable in the sense of Bartle-Dunford-Schwartz with respect to all members of $\mcS_N$.  The integral on a set $E\in\vS$ is  then defined  by the formula
\[
(s)\int_Ef\,dN:=\ov{\biggl\{ {\scriptstyle(BDS)}\int_Ef\,dn\colon n\in\mcS_N\biggr\}}.
\]
The above definition was suggested by Kandilakis \cite{ka} for $cwk(X)$-valued multimeasures and a  Banach space. In that case the set in the braces was closed and
the additional closure is superfluous.
The symbol $(s)$  used here indicates that the integral is constructed using the selections of $N$.
\end{itemize}
\end{deff}
\begin{rem}\label{rem1}
\rm \mbox{ }
\begin{itemize}
	\item A multimeasure $N: \vS \to c(X)$ is called
	\textit{  rich } if $N(A) = \overline{\{n(A), n \in S_N \}}$.
	By a result of Cost\'{e}, quoted in \cite[Theorem 7.9]{hess}, this is verified for  multimeasures which take as their values  weakly compact convex subsets of a Banach space
	 $X$, or $cb(X)$-valued when $X$ is a Banach space possessing the RNP (\cite[Th\'{e}or\`{e}me 1]{co}).
	\item
	 If  $N$ is a $cwk(X)$ valued multimeasure of bounded variation and $X$ is a Banach space, then $\mcS_N\neq\emp$ (see for example \cite{GT,Hi}).  Moreover, by
	 \cite[Theorem 3.2]{ka},  $(s)\dint_Ef\,dN \in cwk(X)$.
	\item If $f,g$ are Aumann-BDS-integrable, then
	$$(s)\int_E(f+g)\,dN\subseteq (s) \int_Ef\,dN \oplus (s) \int_Eg\,dN\,.$$
	The example of $g=-f$ shows that the equality fails in general. In particular, the integral in not additive.
	\item
	 We know from \cite[Theorem 3.3]{ka} that in case of a multimeasure $N$ with
	 $cwk(X)$ values in a Banach space $X$ and bounded $f$, we have for every $x'\in{X'}$ and every $E\in\vS$
\begin{linenomath*}
\label{f2}
\begin{eqnarray}
	s\biggl(x', (s) \int_Ef\,dN\biggr)&=&
	\int_Ef^+\,ds(x',N)+\int_Ef^-\,ds(-x',N)\label{e2}\\
	&=&\int_Ef^+\,ds(x',N)+\int_Ef^-\,ds(x',-N)\notag\\
	&\stackrel{in\; general}{\neq}&\int_Ef^+\,ds(x',N)-\int_Ef^-\,ds(x',N) \notag
	\\ &=&
	\int_Ef\,ds(x',N)\,. \notag
\end{eqnarray}
\end{linenomath*}	
	But it follows from that proof that the  equality (2) holds true in each case when ${\mathcal S}_N\neq\emp$, that is even when  $X$ is a
 	locally convex space and   $N$ is not weakly compactly valued. \\
	In particular, if $f\geq 0$, then
\begin{equation}\label{e6}
	s\biggl(x', (s) \int_E f\,dN\biggr)=\int_Ef\,ds(x',N)
\end{equation}
	and the integral is additive for non-negative functions. Moreover, it follows from (\ref{e2}) that the integral is a multimeasure, for every { Aumann-BDS}-integrable $f$.
\end{itemize}
\end{rem}
As noticed in \cite[Remark 3.2]{wzw}, if $f$ is negative, then the equality (\ref{e6}) may fail.
In fact, according to (\ref{e2}),  we have then
\begin{eqnarray*}
s\biggl(x',(s)\!\int_E(-1)\,dN\biggr)&=&\int_E1\,ds(-x',N)=s(-x',N(E))\notag\\
&\stackrel{in\; general}{\neq}&-s(x',N(E))=\int_E-1\,ds(x',N)\,.
\end{eqnarray*}
 Now we  would like to define an integral with respect to a multimeasure that would be independent of selections of the multimeasure but would be
 consistent with earlier definitions via selections. We know already (\cite{wzw,ka}) that if $X$ is a Banach space and a non-negative $\theta:\vO\to\R$ is integrable with
 respect to $N:\vS\to{cwk(X)}$, then for every $E\in\vS$ there exists $W_E\in{cwk(X)}$ such that for every $x'\in{X'}$  holds true the equality
$s(x',W_E)=\dint_E\theta\,ds(x',N).$\\

 We take this property as the definition of the integral of a non-negative function with respect to an arbitrary multimeasure $N:\vS\to{c(X)}$.
  Our approach is close to  Lewis' \cite{Lew} functional equivalent definition of an integral with respect to a vector measure.
 It is worth to remember that the Lewis definition in \cite{Lew} was also  considered by Kluv\'anek in \cite{k2} (see also \cite{MR2419122}).
We call the integral that we are going to define a \textit{ multivalued-Bartle-Dunford-Schwartz integral} since Bartle,
Dunford and Schwartz were the first who considered such a kind of integration in the vector case.

\begin{deff}\label{defiN}
\rm Let $N:\vS\to{c(X)}$ be a multimeasure. If  $f:\vO\to\R$ is  a non-negative  measurable function,
we say that $f$  is {\it multivalued-Bartle-Dunford-Schwartz   integrable  with respect to $N$}  (shortly $BDS_m$-\textit{ integrable with respect to $N$})
in $c(X)\,\;(cb(X),cwk(X),ck(X))$,
if  for every $E\in\vS$ there exists $C_E\in{c(X)}, (cb(X)$, $cwk(X), ck(X))$ such that for every $x'\in{X'}$
\begin{eqnarray}\label{def-pos}
s(x',C_E)=\int_E f \,ds(x',N).
\end{eqnarray}
We set $C_E:= \dint_E f dN$.\\
We  say that  a measurable $f:\vO\to\R$ is
{\it multivalued--Bartle-Dunford-Schwartz   integrable  with respect to $N$} in $c(X) \;(cb(X), cwk(X),ck(X))$, if $f^+$ and $f^-$ are
 {\it $BDS_m$-integrable  with respect to $N$} in $c(X) \;(cb(X), cwk(X),ck(X))$.
Then, for every $E\in\vS$, we define  the integral of $f$ with respect to $N$ as
 $$ \int_E f\,dN:=\int_E f^+\,dN \oplus \int_E f^-\,d(-N)\,. $$
The above definition is consistent with the property described in   \cite[Theorem 3.3]{ka} (There is a missprint in \cite[Theorem 3.3]{ka}, the sign minus should be replaced by plus).
  Equivalently, $f:\vO\to\R$ is $BDS_m$-integrable  with respect to $N$ in $c(X), (cb(X)$, $cwk(X), ck(X))$,
 if for each $E\in\vS$ there exists $M_f(E)\in c(X)\;(cb(X)$, $cwk(X), ck(X))$ such that for every
 $x'\in{X'}$
\begin{equation}\label{e10}
s(x',M_f(E)) =\int_Ef^+\,ds(x',N)+\int_Ef^-\,ds(x',-N)
\end{equation}
and the right hand side of (\ref{e10}) makes sense. We write then $\dint_E f\,\,dN :=M_f(E)$.
 \end{deff}
  The above definition is consistent with the property described in   \cite[Theorem 3.3]{ka}.
Since $s(x',\pm N):\vS\to(-\infty,+\infty]$ are  $\sigma$-finite measures for every $x'$ then, by (\ref{e10}),
 the set functions $s(x',M_f)$ are  $\sigma$-finite measures
 with values in $(-\infty,+\infty]$. Consequently, $M_f$ is a multimeasure.
 Let us notice that if $N$ is a vector measure (say $N=\nu$) then the right hand side of (\ref{e10}) looks as follows:
$$ s(x',M_f(E)) =\int_Ef\,d\lg{x',\nu}\rg\,.$$
One can easily check that the integral on the right hand side has to be finite  and $M$ is a vector measure  (see   Lemma \ref{L1}).
\begin{rem}\label{rem2.4}
\rm  \mbox{ }
\begin{itemize}
	\item  Assume that ${\mathcal S}_N\neq\emp$ (see for example Remark \ref{rem1}).
	Observe  that if  $f$ is a bounded, measurable,  Aumann-BDS-integrable function  whose integral belongs to $c(X)$ then
	 $f$ is ${BDS}_m$-integrable with respect to $N$  and
\[ (s) \int_E f\,\,dN = M_f (E). \]
	\item
	  If $f$ is a ${BDS}_m$-integrable with respect to $N$ function then
	$M_f (A \cup B) = M_f(A) \oplus M_f (B)$ for every $A,B \in \vS$ with $A \cap B = \emptyset$.
	In fact, for every $x' \in X'$, we have:
\begin{linenomath*}
\begin{eqnarray*}
s\biggl(x',\int_{A \cup B} f \,dN\biggr)&=&
\int_{\vO}(f \chi_{A \cup B})^+\,ds(x',N)+
\int_{\vO}(f \chi_{A \cup B})^-\,ds(x',-N)\\
&=&
\int_{\vO} (f^+ \chi_A + f^+ \chi_B)\,ds(x',N)+
\int_{\vO}(f^- \chi_A +f^- \chi_B)\,ds(x',-N)\\
&=& s\biggl(x',\int_A f\,dN\biggr)+s\biggl(x',\int_B f\,dN\biggr) =s(x', M_f(A) \oplus M_f(B)).
\end{eqnarray*}
\end{linenomath*}
	So, for every $x' \in X'$, it is  $s(x', M_f(A \cup B)) = s(x',M_f(A))+s(x',M_f(B)) = s(x', M_f(A) \oplus M_f(B))$.
\end{itemize}
\end{rem}
\begin{rem}\rm
Let $X$ be a Banach space and $N:\vS\to{cb(X)}$ be a $d_H$-multimeasure. If $f$ is a measurable function such that there exists a sequence
 of simple functions $(f_n)_n$ which pointwise converges to $f$  and the sequences $(\dint_E f_n^{\pm} dN)_n$ are Cauchy in $(cb(X),d_H)$,
  then $f$ is  $BDS_m$ integrable with respect to $N$.\\
Notice first that if $f=  \sum_{i=1}^n a_i 1_{E_i}$ is measurable with non-negative $a_i\,,i=1,\ldots,n$, and $N$ is $cb(X)$-valued, then
$\dint_E f dN = \bigoplus_{i=1}^n a_i N(E \cap E_i)$
 for every $E \in \vS$, since for every $x' \in X'$ the support function is additive with respect to the Minkowski addition. \\
In the general case, if a sequence $(f_n)_n$ of simple functions converges pointwise to $f$, then  $(f_n^{\pm})_n$ converges to $f^{\pm}$.
We consider first $f^+$.  $(f_n^+)_n$ converges poitwise to $f^+$ and  for every $x' \in X'$ and for every $E \in \vS$

  $$  \left(\int_E f_n^+ ds(x',N) \right)_n = \left( s (x', \int_E f_n^+ dN) \right)_n.  $$

  Since $\left( \dint_E f_n^+ dN \right)_n$ is Cauchy in $ (cb(X),d_H)$, by the completeness of the hyperspace
 for every $E \in \vS$  there exists  $M^{+}(E) \in cb(X)$ such that
 \[  \lim_{n \to \infty} d_H \left( \int_E f^{+}_n dN, M^{+}(E) \right) = 0. \]
 We apply now \cite[Lemma 2.3]{Lew} to $f_n^+, f^+$ and $s(x',N)$ and we obtain that $f^+$ is integrable with respect to $s(x',N)$ and
 \[  \lim_{n \to \infty}  \int_E f_n^+ d s(x',N) = \int_E f^+  d s(x',N), \qquad \mbox{uniformly with respect to } \, E \in \vS.  \]
 Since
 \[ d_H \left( \int_E f^{+}_n dN, M^{+}(E) \right) = \sup_{x' \in B_{X'}}
 \left| s(x', \int_E f^{+}_n dN) - s(x',M^+(E)) \right| \]
 we have that
 \[ s(x',M^+(E)) = \int_E f^+ ds(x',N).\]
For the negative part $f^-$ we apply the same construction using $s(x',-N)$; in this way we obtain analogously $M^-(E)$. Finally, considering   $M^+(E) \oplus M^-(E)$ we obtain
\[s(x', M^+(E) \oplus M^-(E)) = \int_E f^+ ds(x',N) + \int_E f^- ds(x',-N). \]
Similarly, if $f$ is a  $BDS_m$-integrable function with respect to a $d_H$-multimeasure $N$, then there exists a sequence $(f_n)_n$ of simple functions  that is
pointwise convergent to $f$ and the sequence $\biggl(\dint_Ef_n\,dN\biggr)_n$ is Cauchy in $(cb(X),d_H)$. See Remark \ref{r4} for the proof.
\end{rem}
\begin{prop}\label{sublinear}
If $N$ is a positive multimeasure, then the   ${BDS}_m$-integral  with respect to $N$  is a sublinear function of its integrands.
\end{prop}
\begin{proof}
 Indeed, assume that $f,g$ are   ${BDS}_m$-integrable  with respect to $N$   and $a,b\geq 0$. Then,
\begin{linenomath*}
\begin{eqnarray*}
s\biggl(x',\int_E(af+bg)\,dN\biggr)&=&\int_E(af+bg)^+\,ds(x',N)+\int_E(af+bg)^-\,ds(-x',N)\\
&\leq&
\int_E(af^+ + bg^+)\,ds(x',N)+\int_E(af^-+bg^-)\,ds(-x',N)\\
&=&a\biggl[\int_Ef^+\,ds(x',N)+\int_Ef^-\,ds(-x',N)\biggr]\\
&+&b\biggl[\int_Eg^+\,ds(x',N)+\int_Eg^-\,ds(-x',N)\biggr]\\
&=& as\biggl(x',\int_Ef\,dN\biggr)+bs\biggl(x',\int_Eg\,dN\biggr).
\end{eqnarray*}
\end{linenomath*}
\end{proof}
The following lemma is essential for our further investigation:
\begin{lem}\label{L1}
 Let $M,N:\vS\to c(X)$  be two multimeasures possessing control measure $\mu$ and such that $M(E)=\dint_E\theta\,dN$, for every $E\in\vS$.
 Then $M$ is pointless if and only if $N$ is pointless. Equivalently, $M$ is a vector measure if and only if $N$ is a vector measure.
\end{lem}
\begin{proof} Assume that $N(E)=\{\kappa(E)\}$ for every $E\in\vS$ and $\kappa:\vS\to{X}$ is a vector measure. Then, we have for each $E\in\vS$ and each $x'\in{X'}$
\begin{linenomath*}\begin{eqnarray}\label{e18}
s(x',M(E))&=& \int_E\theta^+\,ds(x',N)+\int_E\theta^-\,ds(x',-N)=
\int_E\theta^+\,d\lg{x',\kappa}\rg+\int_E\theta^-\,d\lg{x',-\kappa}\rg\notag
\\&=&
\int_E\theta^+\,d\lg{x',\kappa}\rg-\int_E\theta^-\,d\lg{x',\kappa}\rg=\int_E\theta\,d\lg{x',\kappa}\rg\,.
\end{eqnarray}
\end{linenomath*}
 By our assumption the integral $\dint_E\theta\,d\lg{x',\kappa}\rg$ exists
 and  has to be finite. If not, then the equality  $\dint_E\theta\,d\lg{x',\kappa}\rg=+\infty$ yields $s(-x',M(E))=-\infty$, what is impossible.
 Thus, $M$ has only bounded sets as its values.
In such a case  the expression on the right hand side of (\ref{e18}) is a linear function on $X'$. Hence the same holds true for $x'\longrightarrow s(x',M(E))$.
But that means that $M$ is a vector measure. \\
A similar situation takes place if  $M(E):=\{\nu(E)\}\,,\;E\in\vS$, where $\nu$ is a vector measure.
By the assumption, we have

$$
\forall\;E\in\vS,\quad \forall\;x'\in{X'},\qquad
\lg{x',\nu(E)}\rg=\int_E\theta^+\,ds(x',N)+\int_E\theta^-\,ds(x',-N)\,.
$$

Let $A:=\{\omega:\theta(\omega)> 0\}$. By the classical Radon-Nikodym theorem there exists a function  $ f_{x'}$  such that
\be\label{e14}
\forall\;E\in\vS_A,\quad \forall\;x'\in{X'}, \qquad
s(x',N(E))=\int_E f_{x'}\,d\mu\,.
\ee
If $E\in\vS_A$, then

$$
\forall\;E\in\vS_A,\quad \forall\;x'\in{X'},\qquad
\lg{x',\nu(E)}\rg=\int_E\theta^+\,ds(x',N)\,.
$$

Since for every $E\in\vS_A$,  the function $x'\longrightarrow \lg{x',\nu(E)}\rg$ is linear, the same holds true for
$x'\longrightarrow \dint_E\theta^+\,ds(x',N)$. Consequently, if $a,b\in\R$ and $x',y'\in{X'}$, then
\begin{linenomath*}
\begin{eqnarray*}
\int_E\theta^+f_{ax'+by'}\,d\mu&=&\int_E\theta^+\,ds(ax'+by',N) = \\
&=&\int_E\theta^+\,d[a \cdot s(x',N)+b \cdot s(y',N)]=\int_E\theta^+[af_{x'}+bf_{y'}]\,d\mu\,.
\end{eqnarray*}
\end{linenomath*}
As $E\in\vS_A$ is arbitrary and $\theta^+|_A>0$, we obtain the equality
$$ f_{ax'+by'}=af_{x'}+bf_{y'}\quad\mu-a.e. $$
 Then  $x'\longrightarrow s(x',N(E))$ is linear by (\ref{e14}) and this proves that $N$ is a vector measure on $\vS_A$. Similarly for $A^c$.
\end{proof}
\begin{prop}\label{l-pos}
 Let   $M,N:\vS\to {c(X)}$ be two multimeasures and  assume  that  $N$ is pointless.  If    $\theta:\vO\to\R$  is a measurable function such that
  for each $E\in\vS$ and $x' \in X'$
\begin{equation}\label{e17}
	s(x', M(E)) = \int_E \theta\,ds(x', N)\,,
\end{equation}
then $\theta$ is non negative $N$-almost everywhere.
\end{prop}
\begin{proof}
 Let $\theta=\theta^+-\theta^-$ and let $H:=\{\omega \in \vO: \theta (\omega)^->0\} \in \vS$. Then, we obtain for every $E\in\vS_H$ the equality
$$s(x', M(E))=\int_{E} - \theta^-\,ds(x', N)\,.$$
It is enough to prove that $N(H) = \{ 0\}$. We suppose, by contradiction, that $N(H) \neq  \{ 0\}$.
 Then $x'\longrightarrow s(x', M(E))$ is sublinear for each $E\in\vS_H$ and  also $x'\longrightarrow -s(x', M(E))$ is sublinear, because
 $-s(x', M(E))=\int_{E} \theta^-\,ds(x', N)$. Hence $x'\rightarrow s(x', M(E))$ is linear.  Therefore
$$ 0=s(0, M(E))=s(x'-x', M(E))=s(x', M(E))+s(-x', M(E))$$
 and so $s(-x', M(E))=-s(x', M(E))\neq\pm\infty$,   what yields the linearity of $x'\longrightarrow s(x', M(E))$.
 But that is possible only if $M(E)$ is one point set. This however forces $N$ to be a vector measure on $\vS_H$, what contradicts the pointlessness of $N$.
\end{proof}
\begin{rem}\label{r2}
 \rm It follows from Proposition \ref{l-pos} that an integral defined by the equality (\ref{e17}) does not present the proper approach to integrability with respect to a  pointless multimeasure, since only non-negative functions could be integrable.\\
 If in Proposition \ref{l-pos} $N$ restricted to an element $F\notin\mcN(N)$ is a vector measure, then  $M$ restricted to $F$ is a vector measure
 (see Lemma \ref{L1}) and $\theta|_F$ is not necessarily non-negative.
\end{rem}
\section{Control measures}\label{control}
Our aim is to determine when a multimeasure $M$ can be seen as an integral of a scalar function with respect to a given multimeasure $N$.
We obtain our results under the assumption of the existence of control measures for the  considered  multimeasures.
In the case of $d_H$-multimeasures control measures always exist (see \cite{bds}), but in the  case of an arbitrary multimeasure this is not obvious. The subsequent theorem describes completely the class of Banach spaces where every multimeasure is a $d_H$-multimeasure.
\begin{thm}\label{t5}
Every $cb(X)$-valued multimeasure is a $d_H$-multimeasure if and only if $X$ does not contain any isomorphic copy of $c_0$.
\end{thm}
\begin{proof}
If $c_0\nsubseteq{X}$ isomorphically, then the assertion is proved in \cite[Proposition 4.1]{jca2020}. If $c_0$ can be isomorphically embedded into $X$, then \cite[Example 3.6 and 3.8]{mu21} are two examples of $cb(c_0)$-valued multimeasures which are not $d_H$-multimeasures.
\end{proof}
The proofs below provide a characterisation of multimeasures possessing a control measure in the language of the countable chain condition.
The proofs are related to those given in  \cite{mu20}, where vector measures were under consideration.
\begin{deff}\rm
A multimeasure $M:\vS\to{c(X)}$ satisfies the \textit{ countable chain condition} (ccc) if each family of pairwise disjoint not $M$-null sets is at most countable.
\end{deff}
\begin{lem}\label{l3}
Assume that $M,N:\vS\to{c(X)}$ are two multimeasures such that $M\ll{N}$. If $N$ satisfies (ccc), then for every $A\in\vS\smi\mcN(M)$ there exists
$B\in\vS\smi\mcN(M)$ such that $B\subset{A}$ and $N\ll{M}$ on $\vS_B$.
\end{lem}
\begin{proof}
Assume that there is a set $A\in\vS\smi\mcN(M)$ such that for every $B\in\vS_A\smi\mcN(M)$ there exists $D\in[\mcN(M)\cap{B}]\smi\mcN(N)$.
The lemma of Kuratowski-Zorn and (ccc) give the existence of at most countable maximal family $\{D_n\}_n$ of pairwise disjoint sets
$D_n\in[\mcN(M)\cap{A}]\smi\mcN(N)$.\\
 One can easily check that $\bigcup_nD_n\in\mcN(M)$. Since $A\notin\mcN(M)$, we have $A\smi\bigcup_nD_n\notin \mcN(M)$,
 but this contradicts the maximality of $\{D_n\}_n$.
\end{proof}
\begin{lem}\label{l5}
If $N:\vS\to{c(X)}$ is a multimeasure satisfying (ccc), then there exists at most countable family $\{x_n':n\in\N\}\subset X'$ satisfying the equality
\begin{equation}\label{e12}
\bigcap_n\mcN[s(x_n',N)]=\bigcap_{x'\in{X'}}\mcN[s(x',N)]\,.
\end{equation}
\end{lem}
\begin{proof}
 It follows from Lemma \ref{l3} that for every $x'\in{X'}$ there exists $D_{x'}\in\vS\smi\mcN[s(x',N)]$ such that
$N\ll{s(x',N)}$ on $\vS_{D_{x'}}$. The lemma of Kuratowski-Zorn and (ccc) guarantee existence of at most countable maximal family $\{D_n\}_n$ of disjoint
 sets $D_n$  corresponding to measures $s(x_n',N)$.\\
  Let $D=\bigcup_n D_n$.
If $A\in\vS$ and $A\cap{D}=\emp$, then clearly $A\in \mcN[s(x',N)]$, for every $x'\in{B_{X'}}$.\\
 Let now $A$ be an $s(x_n',N)$-null set, for every $n$. We have $A\cap{D_n}\in\mcN(N)$ for all $n$ and so $A\cap{D_n}\in\mcN[s(x',N)]$ for every $x'$.
  Consequently, $A\cap{D}\in\mcN[s(x',N)]$  for every $x'$. Hence,
$$A=(A\smi{D})\cup(A\cap{D})\in\mcN[s(x',N)]\,.$$
That proves (\ref{e12}).
\end{proof}
\begin{thm}\label{T1}
A multimeasure $N:\vS\to{c(X)}$ has a finite control measure if and only if it  satisfies (ccc).
Then, there exists a control measure that is equivalent to $N$.
\end{thm}
\begin{proof}
 It is obvious that the existence of a control measure yields (ccc) of $N$. So assume that $N$ satisfies (ccc). For each $x'\in{X'}$ let
 $\nu_{x'}:\vS\to[0,+\infty)$ be a measure equivalent to $s(x',N)$.
By Lemma \ref{l5} there exist $x_n'\in{X'}\,,n\in\N$, such that
$$\mcN(N)=\bigcap_{x'\in{X'}}\mcN[s(x',N)]=\bigcap_n\mcN[s(x_n',N)]=\bigcap_n\mcN(\nu_{x_n'})\,.$$
The measure
\begin{eqnarray}\label{eq:control}
\mu(E):=\sum_{n=1}^{\infty}\frac{1}{2^n}\frac{\nu_{x_n'}(E)}{1+\nu_{x_n'}(\vO)}
\end{eqnarray}
is the required control measure for $N$ that is  equivalent to $N$.
\end{proof}
\begin{thm}\label{esempi}
  If $X'$ is weak$'$-separable, then every multimeasure $M:\vS\to{c(X)}$ satisfies (ccc).
\end{thm}
\begin{proof} Let $\{x_n':n\in\N\}$ be  weak$'$-dense in $X'$. Suppose there exists a multimeasure $M:\vS\to{c(X)}$ and an uncountable family
$\{B_{\alpha}\in\vS: \alpha\in{\mathbb{A}}\}$ of pairwise disjoint sets with $M(B_{\alpha})\neq\{0\}$.
 Without loss of generality we may assume that the family is ordered by the ordinals less than $\omega_1$. Due to the countability of the weak$'$-dense set, there
 exists $\beta<\omega_1$ such that $s(x_n',M(B_{\alpha}))=0$,
 for every $\alpha>\beta$ and every $n\in\N$. Hence, if $\alpha>\beta$ and $x\in{B_{\alpha}}$, then $\lg{x_n',x}\rg\leq 0$ for every $n$.
 It follows that $\lg{x',x}\rg\leq0$ for every $x'\in{X'}$, if $x\in{B_{\alpha}}$ and $\alpha>\beta$. This is of course impossible for $x\neq0$.
 \end{proof}
 The following recent result describes Banach spaces with all $cb(X)$-multimeasures possessing control measures. 
\begin{thm} \cite{Ro}
 If $X$ is a Banach space not containing any isomorphic copy of $c_0(\omega_1)$, then each multimeasure $M:\vS\to{cb(X)}$ admits a control measure.
 \end{thm}
\section{Arbitrary multimeasures}\label{s-arbitrary}
Now we start by examining the problem when a multimeasure $M$ can be represented as an integral of a scalar function with respect to a given multimeasure $N$, for suitable multisubmeasures this problem was also faced in \cite{ccgs2018} for the Gould integral.\\
In the  case of  vector measures  with values in a locally convex space a Radon-Nikod\'{y}m theorem for the  Bartle-Dunford-Schwartz integral was obtained by Musia{\l} in \cite{mu}.
If $Y$ is a  locally convex space and $\nu, \kappa:\vS\to Y$ are vector measures, then   the following  definitions  were formulated  in \cite{mu}:
\begin{description}
\item[usac)] $\nu$ is \textit{  uniformly scalarly absolutely continuous} (usac) with respect to  $\kappa$, if for each $\ve>0$ there exists $\delta>0$
such that for each $y'\in{Y'}$ and each $E\in\vS$, the inequality $|y'\kappa |(E)<\delta$ yields
 $|y' \nu|(E)<\ve$. We denote it by $\nu \lll \kappa$.
\item[usd)]  $\nu$ is \textit{  uniformly scalarly dominated } (usd) by  $\kappa$,
if there exists  $b\in\R^+$ such that
$\forall \;y'\in{Y'}, \, \forall\;E\in\vS$ it is $
|y' \nu|(E) \leq b\, |y'\kappa |(E).$
\item[sub)]
 $\nu$ is  \textit{ subordinated} to  $\kappa$, if
there exists  $d\in\R^+$ such that for every $E\in\vS$
$$\nu (E)\in  d\,\ov{aco} \, \mcR (\kappa_E).$$
\end{description}
 We say that a vector measure $\nu:\vS\to Y$ has  \textit{ locally} a property  with respect to a vector measure $\kappa:\vS\to Y$,
  if for each $E\in\vS\smi{\mathcal N}(\kappa)$ there exists $F\subset{E}$ with $F\in\vS\smi{\mathcal N}(\kappa)$  such that $\nu$ has this property
   with respect to $\kappa$ on the set $F$.\\

In order to find proper conditions guarateeing the differentiation of an arbitrary $cb(X)$-valued multimeasure $M$ with respect to $N$,
 we must adapt properly the definitions given for vector measures.
\begin{deff}\label{strong}
\rm
Given  two  multimeasures $M,N:\vS\to{cb(X)}$ we say that:
\begin{description}
\item[(usac)]
$M$ is   \textit{  uniformly scalarly absolutely continuous } with respect to  $N$
 ($usac$ or $M\lll N$),  if there exists $A\in\vS$ such that
\begin{eqnarray}\label{eq-susac}
&& \forall\;\ve>0,\;\, \exists\;\delta>0:\quad \forall\;\alpha,\beta\in\R, \;\forall\;x',y'\in X', \;\forall\;E\in\vS\;\\ \nonumber
&&\left[
|\alpha s(x',N)+\beta s(y',N)|(E\cap{A})+|\alpha s(x',-N)+\beta s(y',-N)|(E\cap{A^c})\leq\delta \right] \Rightarrow\\ \nonumber
&&\Rightarrow|\alpha s(x',M)+\beta s(y',M)|(E)\leq\ve;
\end{eqnarray}
\item[(usd)]
 $M$ is \textit{   uniformly scalarly dominated } ($usd$) by a multimeasure  $N$,
if there exist  $c\in\R$ and $A\in\vS$ such that $\forall \;\alpha, \beta\in\R,\;\forall\;x',y'\in X',\;\forall\;E\in\vS\;$
\begin{eqnarray*}
&& |\alpha s(x',M)+\beta{s(y',M)}|(E)\leq \\
&\leq&{c}|\alpha{s(x',N)}+\beta{s(y',N)}|(E\cap{A})+c|\alpha{s(x',-N)}+\beta{s(y',-N)}|(E\cap{A^c})\,;
\end{eqnarray*}
\item[(uss)] $M$ is \textit{   uniformly scalarly subordinated } ($uss$) to $N$,
 if there exist  $d\in\R^+$ and $A\in\vS$ such that $\forall \;\alpha, \beta\in\R,\;\forall\;x',y'\in X',\;\forall\;E\in\vS\;$
\begin{eqnarray*}
&& \alpha s(x',M (E) )+\beta s(y',M (E) )\in \\
&\in & d\, \ov{aco}\{ [\alpha s(x',N(F\cap{A})) +\beta s(y',N(F\cap{A})) ]\colon F\in\vS_E \} +\\ &+&
 d\, \ov{aco}\{ [\alpha s(x',-N (F\cap{A^c}))  +\beta s(y',-N (F\cap{A^c}))] \colon F\in\vS_E \}\,.
 \end{eqnarray*}
\end{description}
\end{deff}
\begin{rem}\label{r1}
\rm If $M$ and $N$ are vector measures then the   uniform scalar absolute continuity and   the   uniform scalar domination  of $M$ with respect  to $N$
 coincide respectively with the corresponding definitions given for vector measures, therefore we use the same notation for the  two notions.\\

  In case of $uss$ the definition looks different with respect to $sub$ and so we decided to change the name.
They actually are equivalent if we assume that $X$ is a  quasi-complete locally convex space.
This is due to the fact that for every $E \in \vS$ and for every $x' \in X'$
\begin{eqnarray}\label{sub-uss}
&& \ov{aco}\{\lg{x',\mcR(\kappa_E)}\rg\}=\lg{x',\ov{aco}(\mcR(\kappa_E))\rg}
\quad \mbox{with} \\
&& \nonumber \lg{x',\mcR(\kappa_E)}\rg:=\{\lg{x',w}\rg:w\in{\mcR(\kappa_E)}\}.
\end{eqnarray}
Here the  inclusion $\supset$ is a consequence of the continuity of $x' \in X'$:
$$
\lg{x',\ov{aco}(\mcR(\kappa_E))\rg}\subset\ov{\lg{x',aco(\mcR(\kappa_E))}\rg}=\ov{aco}\{\lg{x',\mcR(\kappa_E)}\rg\}\,.
$$
While for the reverse inclusion, we can  observe that  the set $\ov{aco}\,\mcR(\kappa_E)$ is   symmetric and  weakly compact,  due to \cite[Theorem IV.6.1]{kk}.\\
So we have
  $\lg{x',\ov{aco}\,\mcR(\kappa_E)}\rg=[-\alpha,\alpha]$. The inclusion $\mcR(\kappa_E) \subset \ov{aco}\, {\mcR(\kappa_E)}$ yields $\lg{x', {\mcR(\kappa_E)}}\rg \subset
[-\alpha,\alpha]$ and so $\ov{aco}\,\lg{x', {\mcR(\kappa_E)}}\rg\subset [-\alpha,\alpha]$.
\end{rem}
\begin{prop}\label{+forte}  For arbitrary multimeasures $M,N :\vS\to{cb(X)}$  the properties $usac$ and $usd$ are equivalent and $uss$ implies each of them.
\end{prop}
\begin{proof}  {\bf usd) $\Rightarrow$ usac) }  is obvious.\\
{\bf (usac)  $\Rightarrow$   (usd)}
Let  $A \in \vS$,  $\ve>0$ and $\delta>0$ be    such  that $|\alpha s(x',N)+\beta s(y',N)|(E)<\delta$ implies $|\alpha{s(x',M)}+\beta{s(y',M)}|(E)<\ve$.\\
If $|\alpha s(x',N)+\beta s(y',N)|(E)=0$, then $M\lll N$ yields
$|\alpha{s(x',M)}+\beta{s(y',M)}|(E)<\ve\,\,\mbox{for every } \ve>0\,. $
 So $usd$ follows.\\
 Suppose now that $|\alpha s(x',N)+\beta s(y',N)|(E)>0$.
Let
\begin{eqnarray*}
 \hat{x}:=\delta{x'}[2|\alpha s(x',N)+\beta s(y',N)|(E)]^{-1}
\quad
 \hat{y}:= \delta{y'}[2|\alpha s(x',N)+\beta s(y',N)|(E)]^{-1}
\end{eqnarray*}
Then
$|\alpha s(\hat{x},N)+\beta s(\hat{y},N)|(E)<\delta $
 and consequently
 $ |\alpha s(\hat{x},M)+\beta s(\hat{y},M)|(E) \leq \ve. $
It follows that, if we take $c= 2 \ve /\delta$, then
\begin{eqnarray*}
|\alpha s(x',M)+\beta s(y',M)|(E) \leq c |\alpha s(x',N)+\beta s(y',N)|(E)\,.
\end{eqnarray*}
In a similar way one obtains the required inequalities for very $E\in\vS_{A^c}$.
That proves $usd$.

{\bf (uss) $\Rightarrow$ (usd) }
 By definition, for every $F\in\vS_E $, we have
\begin{linenomath*}
\begin{eqnarray*}
 \alpha s(x',M (E) )&+&\beta s(y',M (E) )\leq\\  &\leq&
				\sup \, d\, \ov{aco}\{ [\alpha s(x',N(F\cap{A})) +\beta s(y',N(F\cap{A})) ] \} +\\ &+&
				\sup d\, \ov{aco}\{ [\alpha s(x',-N (F\cap{A^c}))  +\beta s(y',-N (F\cap{A^c}))] \}\, \leq \\
&\leq& {d}|\alpha{s(x',N)}+\beta{s(y',N)}|(E\cap{A})+
\\ &+& d|\alpha{s(x',-N)}+\beta{s(y',-N)}|(E\cap{A^c})\,.
 \end{eqnarray*}
 \end{linenomath*}
 So the assertion follows.
\end{proof}

\begin{thm}\label{t1}
Let $M,N :\vS\to{cb(X)}$  be two consistent multimeasures possessing control measures.  Then the following  are equivalent:
\begin{description}
\item[(RN$_b$)]
There exists a bounded measurable function  $\theta:\vO \to\R$ such that for every $E \in \vS$ we have
\begin{equation}\label{e3}
M(E)=\int_E\theta\,dN\,;
\end{equation}
\item[(\ref{t1}.i)]
$M$ is   uniformly scalarly dominated by $N$;
\item[(\ref{t1}.ii)]
 $M$ is uniformly scalarly absolutely continuous  with respect to $N$;
\item[(\ref{t1}.iii)]
$M$ is   uniformly scalarly subordinated to $N$.
\end{description}
\end{thm}
\begin{proof}
Let $\mu$ be a finite control measure for $N$. Without loss of generality, we may assume that $\mu$ is also a control measure for $M$.
  Let $H \in \Sigma$ be a set as in the definition of the consistency for $M$ and $N$.
 We divide the proof into the pointless and vector parts.
 \begin{description}
 \item[ (Pointless part)]
 We assume for simplicity that $H=\vO$.
 \begin{description}
\item[(\ref{t1}.i) $\Rightarrow$ (RN$_b$)]
By  the classical Radon-Nikod\'{y}m theorem, for every $x' \in X'$,
  there exist two measurable   real functions  $f_{x'}$ and  $g_{x'}$ such that, $\forall\;E\in\vS,\;$
\begin{equation}\label{e5}
 s(x',M(E))=\int_Ef_{x'}\,d\mu\quad\mbox{and}\quad  s(x',N(E))=\int_Eg_{x'}\,d\mu\,.
\end{equation}
\begin{itemize}
\item  Let $A$  be the set satisfying  the definition of     uniform scalar domination of $M$ with respect to $N$.
 If $E\in \vS_A$ and  $\alpha , \beta \in \mathbb{R}$, then
\begin{linenomath*}
\begin{eqnarray*}
\int_E|\alpha{f_{x'}}+\beta{f_{y'}}|\,d\mu&=&
|\alpha{s(x',M)}+\beta{s(y',M)}|(E) \\
&\leq & {c}|\alpha{s(x',N)}+\beta{s(y',N)}|(E) \\
&=& c\int_E|\alpha{g_{x'}}+\beta{g_{y'}}|\,d\mu.
\end{eqnarray*}
\end{linenomath*}
So
$|\alpha{f_{x'}}+\beta{f_{y'}}|\leq c|\alpha{g_{x'}}+\beta{g_{y'}}|\quad\mu-\mbox{a.e. on} \;A.$
Hence   (see also  \cite[Lemma]{mu}),

$$\frac{f_{x'}}{cg_{x'}}=\frac{f_{y'}}{cg_{y'}}\quad\mu-\mbox{a.e. on the set} \quad\{\omega\in\vO:g_{x'}(\omega)g_{y'}(\omega)\neq 0\}\cap{A}\,.$$

According to \cite[Theorem 7.35.2]{za} (see also \cite[Lemma]{mu} for a short proof) there exists a  measurable
  $\theta_1:A\to [-1,1]$   such that $f_{x'}=c\, \theta_1\,g_{x'}\;\mu$-a.e.
 on the set $\{\omega \in A: g_{x'}(\omega)g_{y'}(\omega)\neq 0\}$ for each $x' \in X'$ separately. The equality on the set
 $\{\omega \in A: g_{x'}(\omega)= 0\}$ is obvious (notice that $|f_{x'}|\leq c|g_{x'}|$ a.e. on $A$).
So	it follows that

\begin{equation*}
s(x', M(E))= \int_E c\, \theta_1\,ds(x', N)\,
\end{equation*}

	for each $E\in\vS_A$ and $x' \in X'$.
We can prove that $\theta_1 \geq 0$ $N$-a.e. on $A$, thanks to Proposition \ref{l-pos}.
\item Let now $F \in \vS_{A^c}$.
For every $\alpha , \beta \in \mathbb{R}$ and for every $x', y' \in X'$ we have
\begin{linenomath*}
\begin{eqnarray*}
\int_F |\alpha{f_{x'}}+\beta{f_{y'}}|\,d\mu&=&
|\alpha{s(x',M)}+\beta{s(y',M)}|(F) \\
&\leq & {c}|\alpha{s(x',-N)}+\beta{s(y',-N)}|(F) \\ &=&
 c\int_F |\alpha{g_{-x'}}+\beta{g_{-y'}}|\,d\mu.
\end{eqnarray*}
\end{linenomath*}
A similar calculation  gives the equality

$$\frac{f_{x'}}{cg_{-x'}}=\frac{f_{y'}}{cg_{-y'}}\quad\mu-\mbox{a.e. on the set} \quad\{\omega\in A^c :g_{-x'}(\omega)g_{-y'}(\omega)\neq 0\}$$

and then

\begin{equation*}
s(x', M(F))= \int_F c\, \theta_2\,ds(x', -N)\,.
\end{equation*}

   As before, $\theta_2:A^c\to\R$ is non-negative $N$-a.e.. Thus, set $\theta=c\theta_1-c\theta_2$, for every $E \in \vS$
\begin{linenomath*}
\begin{eqnarray*}
s(x',M(E)) &=& s(x',M(E \cap A)) + s(x',M(E \cap A^c)) =\\
&=& \int_E \theta^+\,ds(x', N)+\int_E \theta^-\,ds(x', -N)\,,
\end{eqnarray*}
\end{linenomath*}
what means that $\theta$ is a Radon-Nikod\'{y}m derivative of $M$ with respect to $N$.
\end{itemize}
\item[( RN$_b$) $\Rightarrow$ (\ref{t1}.iii)]
We set
 $A:=\{ \omega \in \vO: \theta(\omega) \geq0\}$,  and $d > \sup_{\omega \in \vO}|\theta(\omega)|$.
Since $\theta$ is
$BDS_m$-integrable with respect to $N$, by Definition \ref{defiN} we have for each $E\in\vS$ the equality
\begin{eqnarray*}
\alpha s(x',M(E\cap{A}))+\beta s(y',M(E\cap{A}))
&=&\int_{E\cap{A}}\theta^+\,d[\alpha s(x',N)+\beta s(y',N)]\,.
\end{eqnarray*}
Let $x',y',\alpha,\beta$ be fixed and let $B,C$ generate the Hahn decomposition of
$\alpha s(x',N)+\beta s(y',N)$. Then,
\begin{linenomath*}
\begin{eqnarray*}
&&\alpha s(x',M(E\cap{A\cap{B}}))+\beta s(y',M(E\cap{A\cap{B}}))\\
&\in& \ov{co}\, [\theta^+(E\cap{A}\cap{B})] \cdot [\alpha s(x',N(E\cap{A\cap{B}}))+\beta s(y',N(E\cap{A\cap{B}}))]
\end{eqnarray*}
\end{linenomath*}
Similarly,
\begin{linenomath*}
\begin{eqnarray*}
&&\alpha s(x',M(E\cap{A\cap{C}}))+\beta s(y',M(E\cap{A\cap{C}}))\\
&\in& \ov{co}\, [\theta^+(E\cap{A}\cap{C})]\cdot [\alpha s(x',N(E\cap{A\cap{C}}))+\beta s(y',N(E\cap{A\cap{C}}))]
\end{eqnarray*}
\end{linenomath*}
All together yields
\begin{linenomath*}
\begin{eqnarray*}
&&\alpha s(x',M(E\cap{A}))+\beta s(y',M(E\cap{A}))\\
&\in& \ov{co}\, [\theta^+(E\cap{A})] \cdot \Big( [\alpha s(x',N(E\cap{A\cap{B}}))+\beta s(y',N(E\cap{A\cap{B}}))] \\
&+&[\alpha s(x',N(E\cap{A\cap{C}}))+\beta s(y',N(E\cap{A\cap{C}}))] \Big)\\
&=& \ov{co}\, [\theta^+(E\cap{A})] \cdot [\alpha s(x',N(E\cap{A}))+\beta s(y',N(E\cap{A}))]\\
&\subset& d\,\ov{aco}\, \{  \alpha s(x',N(F\cap{A})) +\beta s(y',N(F\cap{A}))  \colon F\in\vS_E \}.
\end{eqnarray*}
\end{linenomath*}
In the same way one obtains the inclusion
\begin{linenomath*}
\begin{eqnarray*}
&&\alpha s(x',M(E\cap{A^c}))+\beta s(y',M(E\cap{A^c}))\\
&\subset& d\,\ov{aco}\, \{  \alpha s(x',-N(F\cap{A^c})) +\beta s(y',-N(F\cap{A^c})) \colon F\in\vS_E \}.
\end{eqnarray*}
\end{linenomath*}
 The remaining equivalences follow from Proposition \ref{+forte}.
\end{description}
\item[ (Vector part)]
  We assume now that $H^c=\vO$. In virtue of  Remark \ref{r1},  the equivalence of the four conditions is a consequence of \cite[Theorem 1]{mu}.
  \end{description}
\end{proof}

Similarly to the vector case   we say that a multimeasure $M:\vS\to{cb(X)}$ has  \textit{ locally} a property  with respect to a multimeasure $N:\vS\to{cb(X)}$,
  if for each $E\in\vS\smi{\mathcal N}(N)$ there exists $F\subset{E}$ with $F\in\vS\smi{\mathcal N}(N)$  such that $M$ has the property
   with respect to $N$ on the set $F$. \\

 Using the local properties we obtain:
\begin{thm}\label{t2}
Let $M,N :\vS\to{c(X)}$  be two $\sigma$-bounded consistent multimeasures possessing control measures.   Then the following  are equivalent:
\begin{description}
\item[(RN)]
There exists a measurable function  $\theta:\vO\to \R$ such that
$$\forall\;E\in\vS\qquad M(E)=\int_E\theta\,dN\,; $$
\item[(\ref{t2}.i)]
$M$ is   locally uniformly scalarly dominated by $N$;
\item[(\ref{t2}.ii)]  $M$ is locally uniformly scalarly absolutely continuous  with respect to $N$;
\item[(\ref{t2}.iii)]
$M$ is locally uniformly scalarly subordinated to $N$.
\end{description}
\end{thm}
\begin{proof} Assume that $\vO=\bigcup_n\vO_n\cup{B}$, where $\vO_n$-s are pairwise disjoint,  $B\in\mcN(\mu)$  and the restriction of $M$ and $N$ to
each  $\vO_n$ is $cb(X)$-valued. First  we assume that $M,N$  are $cb(X)$-valued on $\vS$. \\
Analogously to Theorem \ref{t1} we have to divide the proof into the pointless and vector parts. The vector part follows immediately from \cite[Theorem 2]{mu},
so we prove here only the pointless part.
\begin{itemize}
\item The equivalences {\bf (\ref{t2}.i)} $\Leftrightarrow$ {\bf (\ref{t2}.ii)} $\Leftrightarrow$ {\bf (\ref{t2}.iii)} follow from the corresponding equivalences
  in Theorem \ref{t1}.
\item Also   {\bf (RN)} $\Rightarrow$ {\bf (\ref{t2}.i)}  is obvious (it is enough to take  for each $E\in\vS\smi{\mathcal N}(N)$ a set  $F\subset{E}$ such that
 $F\in\vS\smi{\mathcal N}(N)$ and $\theta$ is bounded on $F$).
\item
Now we are going to prove that {\bf (\ref{t2}.i)} $ \Rightarrow $ {\bf (RN)}. \\
Let    $\mu$ be a control measure  for $N$.
By definition there exist $E \in \vS \setminus \mathcal{N}(\mu)$  such that $M$ is scalarly dominated on $E$ by $N$. We denote by
\begin{linenomath*}
\begin{eqnarray*}
\mathcal{H}_1&:=&\{E\in\vS \smi{\mathcal N}(\mu) \colon \exists \;A_E\in\vS_E\quad
\forall\;F\in\vS_E, \,\forall\;x',y'\in{X'}, \; \forall\;\alpha,\beta\in\R, \\
&&|\alpha{s(x',M)}+\beta{s(y',M)}|(F)\\
&\leq&
|\alpha{s(x',N)}+\beta{s(y',N)}|(F\cap{A_E})+|\alpha{s(x',-N)}+\beta{s(y',-N)}|(F\smi{A_E})\}
\,.
\end{eqnarray*}
\end{linenomath*}
By the completeness of the algebra $\vS / {\mathcal N}(\mu)$ there is $E_1 \in \vS$ such that its equivalence class is the least upper bound of
$\mathcal{H}_1$ in $\vS /{\mathcal N}(\mu)$.
If $\mathcal{H}_1$ is empty we choose $E_1 = \emptyset$.
 \\
Then we consider   the class $\mathcal{H}_2$ of all sets $E\in\vS_{\vO \setminus E_1} \smi{\mathcal N}(\mu) $ in which we have the
   scalar domination with respect to  $c=2$ and we choose analogously $E_2$. \\
After $n$ steps we have already sets  in this way $E_1,\ldots,E_n$ and $A_{E_1}\subset{E_1},\ldots,A_{E_n}\subset{E_n}$ such that for each
 $k \leq n$ the inequality
\begin{linenomath*}
\begin{eqnarray*}
&&|\alpha{s(x',M)}+\beta{s(y',M)}|(F)\\
&\leq&
k \left\{
|\alpha{s(x',N)}+\beta{s(y',N)}|(F\cap{A_{E_k}})+
|\alpha{s(y',-N)}+\beta{s(x',-N)}|(F\smi{A_{E_k}}) \right\}
\end{eqnarray*}
\end{linenomath*}
 holds true for all $F \in \vS_{E_k}$, for all $\alpha,\beta\in\R$ and for all $x',y'\in{X'}$.  Then we construct $E_{n+1}$ such that its
 equivalence class is the least upper bound of
\begin{linenomath*}
\begin{eqnarray*}
\mathcal{H}_{n+1} &:=& \{E\in\vS_{\vO \smi \bigcup_{k=1}^n E_k}
\smi{\mathcal N}(\lambda)
 \colon \exists \;A_E\in\vS_E\;
 \forall\;F\in\vS_E, \forall\;x',y'\in{X'},\;\forall\;\alpha,\beta\in\R \\
&&|\alpha{s(x',M)}+\beta{s(y',M)}|(F)
\leq (n+1)|\alpha{s(x',N)}+\beta{s(y',N)}|(F\cap{A_E})\\
&+&(n+1)|\alpha{s(x',-N)}+\beta{s(y',-N)}|(F\smi{A_E})\}\,.
\end{eqnarray*}
\end{linenomath*}
in $\vS / {\mathcal N}(\mu)$. A few first sets $E_k$ may be of measure zero  but then we meet the first set $E_k$ of positive measure.
 In this way  we obtain  a sequence (possibly finite) $(E_n)$  of sets with $E_n \in \mathcal{H}_{n}$.\\
 Without loss of generality we may assume that the sets  cover all $\vO$.
  Now  we  apply Theorem \ref{t1}  to each set $E_n$ and we get a bounded measurable   function $\theta_n$ such that
\begin{equation}\label{e8}
\forall\;E\in\vS_{ E_n},\;  \forall\;x'\in{X'},\quad
 s(x',M(E))=\int_E\theta_n^+\,d s(x',N)+\int_E\theta_n^-\,d s(x',-N)\,.
\end{equation}
Let  $\theta=\sum_{n=1}^{\infty} \theta_n \chi_{E_n}$ 
and let $x'\in{X'}$ be fixed.
 Then,  $\theta^+=\sum_{n=1}^{\infty} \theta_n^+ \chi_{E_n}$ and  $\theta^-=\sum_{n=1}^{\infty} \theta_n^- \chi_{E_n}$.
By (\ref{e8}) we have for every $E \in \vS$
\begin{linenomath*}
\begin{eqnarray}\label{serie}
\nonumber
s(x',M(E)) &=& \sum_{n=1}^{\infty} s(x',M(E \cap E_n)) = \\&=&
\sum_{n=1}^{\infty} \left(
\int_{E \cap E_n} \theta_n^+\,d s(x',N)+\int_{E \cap E_n} \theta_n^-\,d s(x',-N) \right).
\end{eqnarray}
\end{linenomath*}
Let $\vO^+ :=\{\omega\in \vO: \theta(\omega) \geq 0\},\, \vO^-:= \vO \setminus \vO^+$.
Let $P'_1, P'_2$, $Q'_1,Q'_2 \in \vS$ be two Hahn decomposition for $s(x',N), s(x',-N)$ respectively.
We have
\begin{linenomath*}
\begin{eqnarray*}
s(x',M(E \cap P_1' \cap \vO^+) )&=&
\sum_{n=1}^{\infty} s(x',M(E \cap E_n  \cap P_1' \cap \vO^+)) = \\
&=& \sum_{n=1}^{\infty}
\int_{E \cap E_n \cap P_1' \cap \vO^+} \theta_n^+\,d s(x',N)
\\
s(x',M(E \cap P_2' \cap \vO^+) )&=&
\sum_{n=1}^{\infty} s(x',M(E \cap E_n  \cap P_2' \cap \vO^+)) =
\\ &=&
\sum_{n=1}^{\infty}
\int_{E \cap E_n \cap P_2' \cap \vO^+} \theta_n^+\,d s(x',N)
\end{eqnarray*}
\end{linenomath*}
\begin{linenomath*}
\begin{eqnarray*}
s(x',M(E \cap Q_1' \cap \vO^- ))&=&
\sum_{n=1}^{\infty} s(x',M(E \cap E_n  \cap Q_1' \cap \vO^-)) =
\\ &=& \sum_{n=1}^{\infty}
\int_{E \cap E_n \cap P_1' \cap \vO^-} \theta_n^-\,d s(x',-N) \\
s(x',M(E \cap Q_2' \cap \vO^-))&=&
\sum_{n=1}^{\infty} s(x',M(E \cap E_n  \cap Q_2' \cap \vO^-)) =
\\ &=& \sum_{n=1}^{\infty}
\int_{E \cap E_n \cap P_2' \cap \vO^-} \theta_n^-\,d s(x',-N).
\end{eqnarray*}
\end{linenomath*}
Each of the series closing the above equalities is convergent with all its terms of the same sign. Hence, they are absolutely convergent. It follows that
  for all $E\in\vS$
 \begin{linenomath*}
\begin{eqnarray*}
 s(x',M(E))&=& \sum_{n=1}^{\infty}s(x',M(E\cap E_n)) =
 \int_{E \cap P_1'}\theta^+\,ds(x',N) +
 \int_{E \cap P_2'}\theta^+\,ds(x',N)+\\
 &+&
 \int_{E \cap Q_1'} \theta^-\,ds(x',-N) +
 \int_{E \cap Q_2'} \theta^-\,ds(x',-N) = \\
 &=&
\int_{E}\theta^+\,ds(x',N) +  \int_{E} \theta^-\,ds(x',-N) .
\end{eqnarray*}
\end{linenomath*}
\end{itemize}

Let us consider now the general case. We know already that for each $n\in\N$ there exists a measurable function $\xi_n:\vO_n\to\R$ such that
 \begin{equation}\label{e30}
\forall\;E\in\vS_{ \vO_n}\;  \forall\;x'\in{X'}\; s(x',M(E))=\int_E\xi_n^+\,d s(x',N)+\int_E\xi_n^-\,d s(x',-N)\,.
\end{equation}
Then, we follow the proof presented after formula (\ref{e8}). We have to remember only that we have always $s(x',M(E))>-\infty$ and so each series appearing
in the proof is either divergent to $+\infty$ or convergent.
\end{proof}


\begin{rem}\label{r3} \rm It follows from Lemma \ref{L1} that some introductory assumptions in Theorem \ref{t1} concerning $M$ and $N$
 are necessary. If $M$ is pointless but $N$ is not, then $M$ cannot be represented by a $BDS_m$-integral with respect to $N$.
To construct an example let $X$ be a quasi-complete locally convex space and $N:=\nu:\vS\to{X}$ be a non-atomic vector measure.\\

 If $M$ is defined by the formula  $M(E):=\ov{co}\,\mcR(\nu_E)$ (or $M(E):=\ov{aco}\,\mcR(\nu_E)$), then $\ov{aco}\,\mcR(\nu_E)$ is
 a weakly compact set (see \cite[Theorem IV.6.1]{kk}) and the conditions  $uss$,  $usd$ and $usac$ are fulfilled by $M$ and $N$.
 Indeed the formula (\ref{sub-uss}) in  Remark \ref{r1}, shows that
$$\alpha{s(x',M(E))}+\beta{s(x',M(E))}\in d\,\ov{aco}
\{\lg{x',\mcR(\nu_E)}\rg\}
\,,$$
whenever $E\in\vS$ and $x'\in{X'}$ are arbitrary. It remains to prove that
$$M(E):=\ov{co}\, \mcR(\nu_E)$$
is a multimeasure. $M$ is clearly finitely additive: if $A,B$ are disjoint, then
$$
M(A \cup B) =\ov{co} \,\mcR(\nu_{A \cup B}) = \ov{co} \, (\mcR(\nu_A) + \mcR(\nu_B)) =
 \ov{co} \, \mcR(\nu_A) \oplus  \ov{co} \, \mcR(\nu_B)= M(A) \oplus M(B).
 $$
It follows that for every $x' \in X'$, $s(x',M) : \vS \to \mathbb{R}$ is finitely additive. \cite[Proposition 3.8]{bm} yields the countable additivity of
each $s(x',M)$ and so $M$ is a multimeasure.
\end{rem}

\begin{rem}\label{r30}\rm A particular example of a quasi-complete locally convex space is a conjugate Banach space endowed with the weak$^*$ topology
 $\sigma(X',X)$. Let $cw^*k(X')$ denote the family of all non-empty weak$^*$-compact and convex subsets of $X'$. $M:\vS\to cw^*k(X')$ is
  called a {\it weak$^*$-multimeasure} if $s(x,M(\cdot))$ is a measure, for every $x\in{X}$. It is proved in \cite[Theorem 3.4]{munew} that  $X$ does not
  contain any isomorphic copy of $\ell^{\infty}$ if and only if each weak$^*$-multimeasure $M:\vS\to cw^*k(X')$ is a $d_H$-multimeasure.
The Radon-Nikod\'{y}m theorem for two weak$^*$-multimeasures formulates exactly as Theorem \ref{t2}, one should only remember that now the conjugate
 to $(X',\sigma(X',X))$ is the space $X$ itself.
\end{rem}

\section{$d_H$-multimeasures and their   R{\aa}dstr\"{o}m embeddings}\label{s-dH}
Throughout this section we assume that $X$ is a Banach space and we consider now the case of $d_H$-multimeasures.
Let us recall  that if a Banach space $X$ does not contain any isomorphic copy of $c_0$, then each $cb(X)$-valued multimeasure is a $d_H$-multimeasure (see \cite{jca2020,cdpms2018}).
  Examples of $cb(c_0)$-valued measures that are not $d_H$-measures can be found in \cite[Example 3.6, Example 3.8]{mu21}.
When we write $\sum_{m=1}^na_mx_m'\in{span}{B_{X'}}$, 
we always mean that $x_m'\in{B_{X'}}$ and $a_m\in{\mathbb R}$, for every $m\leq{n}$.\\

Let now $M,N:\vS\to{cb(X)}$  be two $d_H$-multimeasures. We remember that
 $\| s ( \cdot, M(E))  \|_{\infty} = \sup_{x' \in B_{X'}} |s(x',M(E))|$ and,  according to  (\ref{eq:Y}),
  we set $Y := \ell_{\infty} (B_{X'}) $.
By \cite[Theorem 1]{mu}, for $\nu:=j \circ M, \, \kappa= j \circ N$ we have
\begin{thm}\label{t3}
	If $M,N:\vS\to{cb(X)}$  are two $d_H$-multimeasures, then the following  are equivalent:
	\begin{description}
	\item[RN$_j$)] \label{alp}
			There exists a bounded measurable function (measurable function) $\theta:\vO\to\R$ such that for all $E\in\vS$ and  $y'\in \ell_{\infty}'(B_{X'})$		
		$$ \langle{y',j\circ{M}(E)}\rangle=\int_E\theta\,d\lg{y',j\circ{N}}\rg\,;$$		
		\item[(\ref{t3}.i)]		
		 $j\circ{M}$ is (locally) uniformly scalarly absolutely continuous  with respect to $j\circ{N}$;
		\item[(\ref{t3}.ii)]
		$j\circ{M}$ is (locally) uniformly scalarly dominated by $j\circ{N}$\,;
			\item[(\ref{t3}.iii)]
		$j\circ{M}$ is (locally) subordinated to $j\circ{N}$.
	\end{description}
\end{thm}

We observe that even if $M$ and $N$ are $d_H$-multimeasures, the differentiation of $M$ with respect to $N$ is in general not equivalent to differentiation  of
$j\circ{M}$ with respect to $j\circ{N}$.
If $\theta$ is the Radon-Nikod\'{y}m derivative of $M$ with respect to $N$, the representation $j\circ{M}(E)={\scriptstyle (BDS)}\dint_E\theta\,d(j\circ{N})$
 for every $E\in\vS$ may fail  (see the subsequent Example \ref{ex3}). In fact one has to take into account  also integration with respect to the
 measure $j\circ(-N)$. If $N$ is a vector measure, then $j\circ(-N)=-j\circ{N}$ and formally the measure $j\circ(-N)$ is absent
  in the calculations  and the integral with respect to $N$ coincides with its BDS-integral.\\

In the case of  multimeasures that are not vector measures, we have in general $j\circ(-N)\neq-j\circ{N}$ and the presence of the measure $j\circ(-N)$ becomes visible.
Next theorem shows the shape  of $j\circ N$ provided $M$ has the Radon-Nikod\'{y}m derivative (in the sense of the integral investigated in this paper) with respect to $N$.
\begin{thm}\label{p4}
	Let $M,N:\vS\to{cb(X)}$  be two $d_H$-multimeasures and $\theta:\vO\to\R$ be a  measurable function. Then $M(E)=\dint_E\theta\,dN$ for every $E\in\vS$ if and only if for all $E\in\vS$ and for all $y'\in\ell'_{\infty}(B_{X'})$
	\begin{equation}\label{e21}
\lg{y',j\circ{M}(E)}\rg=\int_E\theta^+\,d\lg{y',j\circ{N}}\rg+
\int_E\theta^-\,d\lg{y',j\circ{(-N)}}\rg\,.
\end{equation}
 Equivalently,
$$j\circ{M}(E)={\scriptstyle (BDS)}\int_E\theta^+\,d(j\circ{N})+{\scriptstyle (BDS)}\int_E\theta^-\,d(j\circ(-N))\,.$$
\end{thm}
\begin{proof}
\begin{description}
\item[ $\Rightarrow$ ]
\begin{description}
\item[(\ref{p4}.A)]
 We assume first that $\theta$ is bounded. Let $M(E)=\dint_E\theta\,dN$ for every $E\in\vS$. According to (\ref{e10}) we have then
\begin{linenomath*}
\begin{eqnarray*}
\lg{e_{x'},j\circ{M}(E)}\rg&=&s(x',M(E))=\int_E\theta^+\,ds(x',N)+\int_E\theta^-\,ds(-x',N)\\
&=&  \int_E\theta^+\,d\lg{e_{x'},j\circ{N}}\rg+\int_E\theta^-\,d\lg{e_{-x'},j\circ{N}}\rg = \\
&=&  \int_E\theta^+\,d\lg{e_{x'},j\circ{N}}\rg+\int_E\theta^-\,d\lg{e_{x'},j\circ{(-N)}}\rg.
\end{eqnarray*}
\end{linenomath*}
Since $\theta$ is bounded, it is $j\circ{N}$ integrable (as a BDS-integral) and so there are measures $\kappa_1,\kappa_2:\vS\to \ell_{\infty}(B_{X'})$  such that
$$\lg{y',\kappa_1(A)}\rg= \int_E\theta^+\,d\lg{y',j\circ{N}}\rg \qquad\mbox{whenever}\;y'\in\ell'_{\infty}(B_{X'})$$
and
$$\lg{y',\kappa_2(E)}\rg=\int_E\theta^-\,d\lg{y',j\circ{(-N)}}\rg \qquad\mbox{whenever}\;y'\in\ell'_{\infty}(B_{X'}).$$
Hence, we obtain the equality
$$\lg{e_{x'},j\circ{M}(E)}\rg=\lg{e_{x'},\kappa_1(E)}\rg+\lg{e_{x'},\kappa_2(E)}\rg.$$
But  $\{e_{x'}:x'\in{B_{X'}}\}$ is norming and the  measures $j\circ{M}, \kappa_1$ and $\kappa_2$ have weakly relatively compact ranges.
 It follows that $\forall\;y'\in\ell'_{\infty}(B_{X'}),\;\forall \;E\in\vS$
\begin{eqnarray}
 &&\lg{y',j\circ{M}(E)}\rg= \int_E\theta^+\,d\lg{y',j\circ{N}}\rg+\int_E\theta^-\,d\lg{y',j\circ{(-N)}}\rg.\label{e4}
\end{eqnarray}
Assume now that  the equality (\ref{e4}) is fulfilled. Then, $\forall\;x'\in B_{X'},\;\forall\;E\in\vS\;$
\begin{linenomath*}
\begin{eqnarray*}
s(x',M(E))&=&\lg{e_{x'},j\circ{M}(E)}\rg =
\int_E\theta^+ \,d\lg{e_{x'},j\circ{N}}\rg
+\int_E\theta^-\,d\lg{e_{x'},j\circ{(-N)}}\rg\\
&=&\int_E\theta^+\,ds(x',N)+\int_E\theta^-\,ds(x',-N)=s(x',\int_E\theta\,dN)\,.
\end{eqnarray*}
\end{linenomath*}
\item[(\ref{p4}.B)]
 Suppose now that $\theta\geq 0$ is arbitrary and  $M(E)=\dint_E\theta\,dN$ for every $E\in\vS$.\\
That means that for every $x'\in{X'}$ and $E\in\vS$ we have
$$ s(x',M(E))=\int_E\theta\,ds(x',N)\,.$$
 $\theta$ can be represented as $\theta=\sum_n\theta\chi_{E_n}$, where the sets $E_n$ are pairwise disjoint and each $\theta\chi_{E_n}$ is bounded. According to
 (\ref{p4}.A) part,  we have for all $y'\in\ell'_{\infty}(B_{X'})$
$$\lg{y',j\circ{M}(E\cap{E_n})}\rg= \int_{E\cap{E_n}}\theta\,d\lg{y',j\circ{N}}\rg\,,\quad n\in\N\,.$$
Since for each $y'\in\ell'_{\infty}(B_{X'})$ the set function $\lg{y',j\circ{M}}\rg$ is a measure, we obtain the
 equality $\sum_n\lg{y',j\circ{M}(E\cap{E_n})}\rg=\lg{y',j\circ{M}(E)}\rg$. Since for each $y'\in\ell'_{\infty}(B_{X'})$ the
 set function $\lg{y',j\circ{N}}\rg$ is a measure, we have also
$$\sum_n  \int_{E\cap{E_n}}\theta\,d\lg{y',j\circ{N}}\rg=
\int_E\theta\,d\lg{y',j\circ{N}}\rg\,.$$
\item[(\ref{p4}.C)]
 $\theta$ is arbitrary and   $M(E)=\dint_E\theta\,dN$ for every $E\in\vS$. The proof is obvious.
\end{description}
\item[$\Leftarrow$]
 In (\ref{e21}) one should substitute $e_{x'}$ instead of $y'$.
 \end{description}
\end{proof}
\begin{rem}\label{r4}\rm
Let $f$ be a
$BDS_m$-integrable function and  assume that $N$ is  $cb(X)$-valued  $d_H$-multimeasure.
Then there exists a sequence $(f_n)_n$ of simple functions  that is pointwise convergent to $f$ and the sequence $\biggl(\dint_E f_n\,dN\biggr)_n$ is
Cauchy in $(cb(X),d_H)$. \\
Indeed, let $M(E)=  \int_E f \,dN\,,E\in\vS$. According to Theorem \ref{p4}  $f^+$ and $f^-$ are $j\circ{N}$  integrable  as (BDS) integrals)
 and we have for all $E\in\vS$, with $E \subset supp\,f^+$
 $$
j\circ{M}(E)=  {\scriptstyle (BDS)} \int_E f^+\,d\,(j\circ{N})\,.
$$
In virtue of  \cite[Definition 2.5]{bds} there exists a sequence of  pointwise convergent to $f^+$ simple functions $h_n:supp\,f^+\to[0,\infty)$ such that the sequence
$\biggl( {\scriptstyle (BDS)} \dint_E h_n\,d\,(j\circ{N})\biggr)_n$ is Cauchy in $\ell^{\infty}(B_{X'})$. It follows from
Theorem \ref{p4} that the sequence $\biggl(\dint_E h_n\,dN\biggr)_n$ is Cauchy in $d_H$. We repeat the procedure
 with $f^-$ obtaining a sequence $(g_n)_n$ of simple functions. Setting $f_n:=h_n-g_n$ we find the required sequence.\\
\end{rem}

By Theorem \ref{p4} and Proposition \ref{l-pos} we get
\begin{cor}\label{p1}
	Let $M,N:\vS\to{cb(X)}$ be two consistent $d_H$-multimeasures.	
Moreover let $\theta:\vO\to\R$ be a  measurable function.
 \begin{itemize}
\item If
	  \, $j\circ{M}(E)={\scriptstyle (BDS)}\dint_E\theta\,d (j\circ{N})$ for every $E\in\vS$, then $\theta\geq 0$ $N$-a.e.  and   for all $E\in\vS$ and all $x'\in{X'}$
$s(x',M(E))=\dint_E\theta\,ds(x',N)$.\\
\item And conversely, if $\theta$ is non-negative
and $s(x',M(E))=\dint_E\theta\,ds(x',N)$ for every $E\in\vS$, then
$j\circ{M}(E)={\scriptstyle (BDS)}\dint_E\theta\,d (j\circ{N})$ for every $E\in\vS$.
\end{itemize}
\end{cor}

It is our aim to obtain a Radon-Nikod\'{y}m theorem for  R{\aa}dstr\"{o}m embeddings of multimeasures in terms of the multimeasures themselves, not their R{\aa}dstr\"{o}m embeddings.  To achieve it we modify the  definitions of $usac$, $usd$ and $uss$   in the following way:
\begin{deff}\label{vstrong}
\rm
Given  two  multimeasures $M,N:\vS\to{cb(X)}$ we say that:
 \begin{description}
\item[(s-usac)]
 A multimeasure $M:\vS\to{cb(X)}$ is \textit{ strongly uniformly scalarly absolutely continuous} ($s$-$usac$) \textit{ with respect to a multimeasure}
$N:\vS\to{cb(X)}$, if for each $\ve>0$ there exists $\delta>0$ such that for each  $\sum_{m=1}^na_mx_m'\in{\rm span }B_{X'}$ and each
$E\in\vS$,   if $|\sum_{m=1}^na_ms(x_m',N)|(E)\leq\delta$,  then
$$\biggl|\sum_{m=1}^na_ms(x_m',M)\biggr|(E)\leq\ve.$$
 We denote it by $M \lll_s N$.
\item[(s-usd)] A multimeasure $M:\vS\to{cb(X)}$ is  \textit{ strongly uniformly scalarly dominated} ($s$-$usd$) \textit{ by a multimeasure}  $N:\vS\to{cb(X)}$, if
there exists a positive $d\in\R$ such that for every $E\in\vS$ and every  $\sum_{m=1}^n a_m x_m' \in {\rm span \,}B_{X'} $, one has
$$\biggl|\sum_{m=1}^n a_m s(x'_m,M)\biggr|(E)\leq  d\,
\biggl|\sum_{m=1}^n a_m s(x_m',N)\biggr|(E)\,.$$
\item[(s-uss)] A multimeasure $M$ is \textit{ strongly  uniformly scalarly subordinated to $N$} ($s$-$uss$), if
there exists a positive $d\in\R$ such that for every $E\in\vS$ and every  $\sum_{m=1}^n a_m x_m' \in {\rm span }\,  B_{X'} $, one has
\begin{eqnarray}
\sum_{m=1}^n a_m s(x'_m,M(E))
\in d\,\ov{aco}\,\biggl\{\sum_{m=1}^na_m s(x'_m,N(F))\colon F\in\vS_E\biggr\}\,.
\end{eqnarray}
\end{description}
\end{deff}

Using these strong versions of  uniform scalar absolute  continuity, uniform scalar domination and uniform scalar subordination we are able to prove the following result:
\begin{thm}\label{t4}
	Let $M,N :\vS\to{cb(X)}$  be two consistent $d_H$-multimeasures. Then the following  are equivalent
	\begin{description}
		\item[(RN$_j$)]
		There exists a bounded measurable function (measurable function) $\theta:\vO\to\R$ such that for all $E\in\vS$ and
		$y'\in \ell_{\infty}'(B_{X'})$		
		$$\lg{y', j\circ{M}(E)}\rg =\int_E\theta\,d \lg{y',j\circ{N}}\rg\,;$$
				\item[(\ref{t4}.j)]
		${M}$ is (locally)	strongly uniformly scalarly dominated by  $N$;
		\item[(\ref{t4}.jj)]
	     $M$ is    (locally) strongly uniformly scalarly absolutely continuous with respect to  $N$;
\item[(\ref{t4}.jjj)] $M$ is  (locally)  strongly  uniformly scalarly subordinated  to $N$.
 \end{description}
\end{thm}
\begin{proof}
\begin{description}
\item[(RN$_j$) $\Rightarrow$ (\ref{t4}.jjj)]
By Theorem \ref{t3}   the condition (RN$_j$) implies that	$j\circ{M}$ is $s$-$uss$ to  $j\circ{N}$\,. Assume that $d>0$ is such that
 $$ j\circ{M}(E)\in{d}\,\ov{\rm aco}\{j\circ{N}(F)\colon F\in\vS_E\}
 \in cb(\ell_{\infty} (B_{X'}))
 \qquad\mbox{for every } E\in\vS\,. $$
 This means that for every $y'\in\ell_{\infty}'(B_{X'})$, we have
\begin{eqnarray*}
&&  \langle{y',j\circ{M}(E)}\rangle  \leq
 d  \sup \{ \lg y', z \rg: \,  z \in \ov{\rm aco}\{j\circ{N}(F)\colon  F\in\vS_E\} \,\}
\\
 &=&  d\, \sup s( y', {\rm aco}\{j\circ{N}(F)\colon F\in\vS_E \})\\
 &\leq&
 d\max\{s(y',{\rm co}\{j\circ{N}(F)\colon F\in\vS_E\}), s(y',{\rm co}\{-j\circ{N}(F)\colon  F\in\vS_E\})\}\\
 &=&d\max\{s(y',\{j\circ{N}(F)\colon  F\in\vS_E\}), s(y',\{-j\circ{N}(F)\colon  F\in\vS_E\})\}.
\end{eqnarray*}
So,  if $\{x_1',\ldots,x_n'\}\subset {X'}$ and $a_i\in\R,\,i=1,\ldots,n$, then
\begin{eqnarray*}
 \sum_{i=1}^na_is(x_i',M(E))
&=& \sum_{i=1}^na_i\lg{e_{x_i'},j\circ{M(E)}}\rg=\biggl\lg{\sum_{i=1}^na_i e_{x_i'},j\circ{M(E)}}\biggr\rg\\
&\leq& d\max\biggl\{s\biggl(\sum_{i=1}^na_i e_{x_i'},\{j\circ{N(F)}\colon F\in\vS_E\}\biggr),
\\ &&
s\biggl(\sum_{i=1}^na_i e_{x_i'},\{-j\circ{N(F)}\colon F\in\vS_E\}\biggr)\biggr\}
\\
&=& d \max\biggl\{\sup_{F\in\vS_E}\biggl\lg{\sum_{i=1}^na_i e_{x_i'},j\circ{N(F)}}\biggr\rg,
\sup_{F\in\vS_E}\biggl\lg{\sum_{i=1}^na_i e_{x_i'},-j\circ{N(F)}}\biggr\rg\biggr\}
\\
&=& d \max\biggl\{\sup_{F\in\vS_E}   \biggl[\sum_{i=1}^na_i\lg{e_{x_i'},j\circ{N} (F)}\rg\biggr] ,
\sup_{F\in\vS_E} \biggl[-\sum_{i=1}^na_i\lg{e_{x_i'},j\circ{N} (F)}\rg\biggr] \biggr\}\\
&=& d \max\biggl\{\sup_{F\in\vS_E}   \biggl[\sum_{i=1}^na_is(x_i',N(F))\biggr],
\sup_{F\in\vS_E}\biggl[-\sum_{i=1}^na_is(x_i',N(F))\biggr]\biggr\}\\
&=& d \max\biggl\{\sup_{F\in\vS_E}   \biggl[\sum_{i=1}^na_is(x_i',N(F))\biggr],
-\inf_{F\in\vS_E}\biggl[\sum_{i=1}^na_is(x_i',N(F))\biggr]\biggr\}\\
&=& d \sup\,{\rm \ov{aco}}\biggl\{\sum_{i=1}^na_is(x_i',N(F))\colon F\in\vS_E\biggr\}.
\end{eqnarray*}
Since we have also
$$
-\sum_{i=1}^na_is(x_i',M(E)) \leq d \sup\,\ov{\rm aco}\biggl\{\sum_{i=1}^na_is(x_i',N(F))\colon F\in\vS_E\biggr\}\,,$$
we deduce that
$$\sum_{i=1}^na_is(x_i',M(E)) \in d\,\ov{aco}\biggl\{\sum_{i=1}^na_is(x_i',N(F))\colon F\in\vS_E\biggr\}\,.$$
\item[(\ref{t4}.jjj) $\Rightarrow$ (\ref{t4}.j)]
By the assumption
\begin{linenomath*}
\begin{eqnarray*}
\pm\sum_{i=1}^na_i s(x_i',M(E)) &\leq& d \sup\,{\rm \ov{aco}}\biggl\{\sum_{i=1}^na_is(x_i',N(F))\colon F\in\vS_E\biggr\}\\
&\leq &  d \sup\,{\rm \ov{aco}}\biggl\{\biggl|\sum_{i=1}^na_is(x_i',N)\biggr|(F)\colon F\in\vS_E\biggr\}\\
&\leq& d \biggl|\sum_{i=1}^na_is(x_i',N)\biggr|(E)\,.
\end{eqnarray*}
\end{linenomath*}
Hence, for every set $E \in \vS$ we have
\begin{eqnarray*}
\biggl|\sum_{i=1}^na_i s(x_i',M)(E)\biggr|\leq d \biggl|\sum_{i=1}^na_is(x_i',N)\biggr|(E)
\end{eqnarray*}
and the standard calculation gives the required inequality for the variations.
\item[ (\ref{t4}.j) $\Rightarrow$ (\ref{t4}.jj)] is obvious.
\item[(\ref{t4}.jj) $\Rightarrow$ (RN$_j$)]
  Assume that for each $\ve>0$ there exists $\ve/2>\delta>0$ such that for each  $\sum_{m=1}^na_mx_m'\in{\rm span }\, B_{X'}$ and each $E\in\vS$, we have
  \begin{eqnarray*}
   \biggl|\sum_{m=1}^na_ms(x_m',N)\biggr|(E)\leq\delta, \,\, \Rightarrow \biggl|\sum_{m=1}^na_ms(x_m',M)\biggr|(E)\leq\ve/2.
   \end{eqnarray*}
    It follows that
 \begin{equation}\label{e1}
 \left|\biggl\langle{\sum_{m=1}^na_me_{x_m'},j\circ{N}}\biggr\rangle\right|(E)\leq\delta\Rightarrow
 \left|\biggl\langle{\sum_{m=1}^na_me_{x_m'},j\circ{M}}\biggr\rangle\right|(E)\leq\ve/2\,.
 \end{equation}
Let $E\in\vS$ and $y'\in\ell_{\infty}'(B_{X'})$ be such that  $|\langle{y',j\circ{N}}\rangle|(E)\leq\delta/2$. Then, let $E^+_{y'}\cup{E^-_{y'}}=E$ be the
 Hahn decomposition of $E$ with respect to $\langle{y',j\circ{N}}\rangle$.\\
Since $\{e_{x'}:x'\in{B_{X'}}\}$ is norming for $\ell_{\infty}(B_{X'})$ there exists $z'\in{\rm span}\{e_{x'}:x' \in B_{X'}\}$ such that
$$
|\langle{y'-z',j\circ{N}(E^+_{y'})}\rangle|<\delta/4 \quad \mbox{and}\quad
|\langle{y'-z',j\circ{N}(E^-_{y'})}\rangle|<\delta/4
$$
and
$$
|\langle{y'-z',j\circ{M}(E^+_{y'})}\rangle|<\delta/4 \quad \mbox{and}\quad
|\langle{y'-z',j\circ{M}(E^-_{y'})}\rangle|<\delta/4.
$$
We have then
\begin{linenomath*}
\begin{eqnarray*}
|\langle{z',j\circ{N}}\rangle(E)|&=&
|\langle{z',j\circ{N}(E^+_{y'})}\rangle+\langle{z',j\circ{N}(E^-_{y'})}\rangle|
\\&=&
 [\langle{z',j\circ{N}(E^+_{y'})}\rangle-\langle{y',j\circ{N}(E^+_{y'})}\rangle]\\
&+&
[\langle{z',j\circ{N}(E^-_{y'})}\rangle-\langle{y',j\circ{N}(E^-_{y'})}\rangle]
\\&+&
 [\langle{y',j\circ{N}(E^+_{y'})}\rangle+\langle{y',j\circ{N}(E^-_{y'})}\rangle]<\delta.
\end{eqnarray*}
By the assumption $|\langle{z',j\circ{M}}\rangle|(E)\leq\ve/2$. Now we follow the reverse way:
\begin{eqnarray*}
|\langle{y',j\circ{M}}\rangle|(E) &=&
\langle{y',j\circ{M}(E^+_{y'})}\rangle+\langle{y',j\circ{M}(E^-_{y'})}\rangle
\\&=&
[\langle{y',j\circ{M}(E^+_{y'})}\rangle-\langle{z',j\circ{M}(E^+_{y'})}\rangle]\\
&+&
[\langle{y',j\circ{M}(E^-_{y'})}\rangle-\langle{z',j\circ{M}(E^-_{y'})}\rangle]
\\&+&
[\langle{z',j\circ{M}(E^+_{y'})}\rangle+\langle{z',j\circ{M}(E^-_{y'})}\rangle]\\ &<& \delta+\ve/2<\ve.
\end{eqnarray*}
\end{linenomath*}
That means that $j\circ{M}$ is $s$-$usac$  with respect to $j\circ{N}$ and so, in virtue of Theorem \ref{t3}
$j\circ{M}$ has the derivative.\\
\end{description}

The proof of the local versions is analogous to that  given in Theorem \ref{t2} where the construction is given for the general case of an arbitrary multimeasure.
\end{proof}

\begin{rem} \rm
Theorems \ref{t3} and \ref{t4} allow us to observe that 	$j\circ{M}\lll{j\circ{N}}$  if and only if $M$ is $s$-$usac$
with respect to  $N$. Analogously $j\circ{M}$  is $usd$ by ${j\circ{N}}$ if and only if $M$ is $s$-$usd$ by $N$.
\end{rem}

\begin{rem}\label{r10} \rm
It turns out that the condition $(RN_j)$ implies the following one (that coincides with the sub condition in case of vector measures):
\begin{description}
\item [(sub)]
\qquad$\forall\;E\in\vS\;M(E)\subset {d}\,\ov{\rm aco}\left[\, \mathcal{R} (N_E) \cup \mathcal{R} (-N_E)\right].$
\end{description}

 In fact, by Theorem \ref{t3} $j\circ{M}$ is subordinated to $j\circ{N}$. Assume that $d>0$ is such that, for every $E \in \vS$,
 
$$ j\circ{M}(E)\in{d}\,\ov{\rm aco}\{j\circ{N}(F)\colon F\in \vS_E\} \in cb(\ell_{\infty} (B_{X'})). $$

 This means that for every $y'\in\ell_{\infty}'(B_{X'})$, we have
 \begin{linenomath*}
\begin{eqnarray*}
 \langle{y',j\circ{M}(E)}\rangle  &\leq&
  d\, s(y',\ov{\rm aco}\{j\circ{N}(F)\colon F \in \vS_E\})
  = d\, s(y',{\rm aco}\{j\circ{N}(F)\colon  F \in \vS_E\})\\
 &\leq&d\max\left\{s(y',{\rm co}\{j\circ{N}(F)\colon F \in \vS_E\}), s(y',{\rm co}\{-j\circ{N}(F)\colon F \in \vS_E\})\right\}\\
 &=&d\max \left\{s(y',\{j\circ{N}(F)\colon  F \in \vS_E\}), s(y',\{-j\circ{N}(F)\colon  F \in \vS_E\}) \right\}.
\end{eqnarray*}
\end{linenomath*}
 In particular, if $x'\in B_{X'}$, then
 \begin{linenomath*}
 \begin{eqnarray*}
 s(x',M(E))&=&\langle{e_{x'},j\circ{M}(E)}\rangle\leq  d\, s(e_{x'},{\rm aco}\{j\circ{N}(F)\colon  F \in \vS_E\})\\
 &=&d\max\{s(e_{x'},\{j\circ{N}(F): F \in \vS_E\}),s(e_{x'},\{-j\circ{N}(F):F \in \vS_E\})\}\\
 &=&d\max\{s(e_{x'},\{j\circ{N}(F): F \in \vS_E\}),s\{
 \lg{e_{x'},-j\circ{N}(F)}\rg: F \in \vS_E\})\}\\
 &=& d\max\{s(x',N(F)): F \in \vS_E\}),s\{\lg{x',-N(F)}\rg: F \in \vS_E\})\}\\
 &=& d \max\{s(x', \mathcal{R} (N_E)\,) , s(x', \mathcal{R} (-N_E)\,) \} \leq
 d\, s(x', \mathcal{R} (N_E) \cup \mathcal{R} (-N_E)\, )\\
 &\leq&  d\, s(x', \ov{\rm aco} \left[\mathcal{R} (N_E) \cup \mathcal{R} (-N_E) \right] \,) =
 s(x', d\, \ov{\rm aco} \left[\mathcal{R} (N_E) \cup \mathcal{R} (-N_E) \right]\,).
 \end{eqnarray*}
 \end{linenomath*}
 In particular, if $x\in{M(E)}$, then $\lg{x',x}\rg\leq s(x', d\, \ov{\rm aco} \left[\mathcal{R} (N_E) \cup \mathcal{R} (-N_E) \right]\,)$.
 In virtue of the Hahn-Banach theorem $x\in {d}\,\ov{\rm aco}\left[\, \mathcal{R} (N_E) \cup \mathcal{R} (-N_E)\right]$. Thus,
\begin{eqnarray*}
M(E)\subset {d}\,\ov{\rm aco}\left[\, \mathcal{R} (N_E) \cup \mathcal{R} (-N_E)\right].
\end{eqnarray*}

The example  given in Remark \ref{r3} proves that in case of arbitrary multimeasures  the condition $sub$  is sometimes essentially weaker than $(RN_j)$.
The current example shows that also in case of pointless multimeasures  the  $sub$   does not guarantee the existence of the Radon-Nikod\'{y}m derivative.
It is enough to take a pointless multimeasure $N:\Sigma \to{cwk(X)}$ and define $M:\vS\to{cwk(X)}$ by $M(E):= \ov{\rm aco}\left[\, \mathcal{R} (N_E) \cup \mathcal{R} (-N_E)\right]$.
\end{rem}

Next theorem   provides a Radon-Nikod\'{y}m representation without invoking to the  R{\aa}dstr\"{o}m embedding.

\begin{thm}\label{ultimo}
Let $M,N:\vS\to{cb(X)}$ be two consistent $d_H$-multimeasures. If $M$ is s-usac	(s-usd or s-uss) with respect to $N$,  then there exists a non-negative,
measurable,  bounded,   ${BDS}_m$-integrable  with respect to $N$ function   $\theta:\vO \to \R$   such that	
  $$M(E) = \int_{E} \theta\,  dN, \qquad \forall \, E \in \vS.$$
	 \end{thm}
	 \begin{proof}
	 By Theorem \ref{t4} the  $s$-$usac$ ($s$-$usd$, $s$-$uss$)	 condition on $M$ and $N$ is equivalent to property $(RN_j)$. Then, for every $x' \in B_{X'}$, it is
\[
 \langle e_{x'},j\circ{M(E)} \rangle = \int_E \theta \, d \langle e_{x'},j\circ{N} \rangle\]	
 which is equivalent to
 \[ s(x', M(E)) = \int_E \theta \,  d s(x',N).\]
 Then, by Proposition \ref{l-pos}, $\theta$ is non-negative $N$-a.e.. Now, by (RN$_j$) and Theorem \ref{p4},
 we have that $\theta$ is   ${BDS}_m$-integrable  with respect to $N$  and
 \[ M(E) = \int_E \theta dN.\]
\end{proof}

\begin{cor}\label{c2} Let $M,N:\vS\to{cb(X)}$ satisfy the hypotheses of Theorem \ref{ultimo}.
  If $j\circ{M}$ can be represented as a BDS-integral with respect to $j\circ{N}$ with a non-negative  bounded
 density $\theta$, then also $M$ can be represented as the  ${BDS}_m$-integral of $\theta$ with respect to $N$.
 \end{cor}
\section{Examples}\label{s-ex}
\begin{ex}\label{ex1}  \quad \rm
Let $\mu$ be a finite measure on $(\vO,\vS)$ and $N(E) := [0,\mu(E)]$
be an interval valued multimeasure (for results and applications of this kind of multimeasure see for example \cite{ccgis,anca2} and the references therein).
We can observe that if $f$ is a scalar, bounded measurable function, then $f$ is  $BDS_m$-integrable  with respect to $N$  and
\[ M_f (E) = \left[- \int_E f^- \, d\mu,  \int_E f^+ \, d\mu \right]. \]
\end{ex}
\begin{ex}\label{ex1a}
\rm
 Let let $N$ be as in Example \ref{ex1}
and let  now $M(E):=[0,\nu(E)]$  where $\nu$ is a finite measure on  $(\vO,\vS)$, equivalent with $\mu$.
	Let $\nu(E)=\dint_E\theta\,d\mu$ for all $E\in\vS$. Then there exists a partition $\vO=\bigcup_n\vO_n$ such that $0\leq\theta\leq{n}\;\mu$-a.e. on $\vO_n$.
	
Moreover  let $\xi: \R \to \R$ be defined by: $\xi(a) = a$ if $a > 0$, otherwise $\xi(a) =0$. So  for every $a \in \R\;s(a, N(E)) = \xi(a) \mu(E)$. If $\alpha,\beta\in\R$ and
 $x'=a,\, y'=b\in\R$ then
\begin{eqnarray*}
&& |\alpha s(a,N)+\beta s(b,N)|(E) = |\alpha \xi(a) + \beta  \xi(b)| \mu(E)\quad\mbox{and}\quad \\
&&|\alpha s(a,M)+\beta s(b,M)|(E)  = |\alpha \xi(a) + \beta  \xi(b)| \nu(E)\,.
\end{eqnarray*}
It follows that if $E\in\vS_{\vO_n}$, then
$$
|\alpha s(a,M)+\beta s(b,M)|(E)  = |\alpha s(a,N)+\beta s(b,N)|(E)\cdot \frac{\nu(E)}{\mu(E)}\leq n|\alpha s(a,N)+\beta s(b,N)|(E)\,.
$$
Therefore the multimeasure $M$ is locally $usd$ with respect to the multimeasure $N$ and can be represented as a $BDS_m$-integral with respect to $N$.
One can easily check that
$M(E)=\dint_E\theta\,dN$.
\end{ex}
\begin{ex}\label{ex2}
\rm
Assume that $X$ is a Banach space and consider $([0,1], \mathcal{L},\lambda)$ where $\lambda$ is  Lebesgue measure and $\mathcal{L}$ is
the family of all Lebesgue measurable subsets of $[0,1]$.
Let $f,g: [0,1] \to X$ be two     Pettis  integrable functions and
$\Gamma(t): = \mbox{co} \{0,f(t)\} , \, \Delta(t): =\mbox{co} \{0,g(t)\}$ be the multifunctions determined by $f$ and $g$ respectively.
They are $ck(X)$-valued and  Pettis integrable  (see \cite[Propositions 2.3 and 2.5]{cdpms2019}). We observe that, for every $x' \in X'$,
\begin{eqnarray}
s(x',\Gamma(t)) = \lg{x',f}\rg^+ (t), \qquad s(x',\Delta(t)) = \lg{x',g}\rg^+ (t)\,. \label{uno}
\end{eqnarray}
If $M,N: \mathcal{L} \to cwk(X)$ are the indefinite multivalued Pettis integrals of $f,g$, then, $\forall\;x'\in{X'},\;\forall\;E\in\mcL$,
\begin{eqnarray}\label{e22}
\quad s(x',M(E)) = \int_E  \lg{x',f}\rg^+ d\lambda, \quad\&\quad s(x',N(E)) = \int_E \lg{x',g}\rg^+  d\lambda.
\end{eqnarray}
Assume that there exists  a   measurable scalar function $\theta$ which is $BDS_m$-integrable  with respect to $N$ and
 \[
 M(E) = \int_E \theta \, dN, \qquad \forall\,  E \in \mathcal{L}.
 \]
 Since $M,N$ are positive,  $\theta$ is non-negative. It is a consequence of Definition \ref{defiN} that
 \[
 s(x',M(E))=\int_E\theta\,ds(x',N)\quad \mbox{for all }x'\in{X'}\;\mbox{and }E\in\mcL.
 \]
 Due to (\ref{e22}) we have
 \[\int_E\lg{x',f}\rg^+\,d\lambda=\int_E\theta\,ds(x',N)=\int_E\theta\lg{x',g}\rg^+\,d\lambda\quad \mbox{for all }x'\in{X'}\;\mbox{and }E\in\mcL.
 \]
It follows that for each $x'\in{X'}$ we have $\lg{x',f}\rg=\theta\lg{x',g}\rg$ $\mu$-a.e.
\end{ex}
\begin{ex}\label{ex3}  \quad \rm   Assume that $X$ is a Banach space and $\mu$ is a non-trivial atomless finite measure on $(\vO,\vS)$.
Let $f_,f_2:\vO\to{X}$ be scalarly integrable functions and $r_1,r_2:\vO\to(0,\infty)$ be $\mu$-integrable functions. Following \cite{mu4}  or
\cite[Example 2.13]{cdpms2019}, we define for $i=1,2$  multifunctions $\vG_i:\vO\to{cb(X)}$ by the formulae
 $\vG_i(\omega)=B(f_i(\omega),r_i(\omega))$, where $B(x,\delta)$ is the closed ball with its center in $x$ and of radius $\delta$.
Then $s(x',\vG_i(\omega))=\lg{x',f_i(\omega)}\rg +r_i(\omega)\|x'\|$  and so each $\vG_i$ is scalarly integrable.
If $f_1,f_2$ are Pettis integrable, then each $\vG_i$ is Pettis integrable in $cb(X)$. Moreover,
$$
(P)\int_E\vG_i\,d\mu=B\biggl((P)\int_Ef_i\,d\mu,\int_Er_i\,d\mu\biggr)\quad i=1,2.
$$
Let $M(E):=(P)\dint_E\vG_1\,d\mu$ and $N(E):=(P)\dint_E\vG_2\,d\mu$. One can easily check that $M$ and $N$ are $d_H$-measures.
Suppose that $M(E)= \dint_E\theta\,dN,\;E\in\vS$, i.e.
$$
\forall\;x'\in{X'},\;\forall\;E\in\vS,\, \qquad s(x',M(E))=\int_E\theta^+\,ds(x',N)+\int_E\theta^-\,ds(x',-N)\,.
$$
That yields the equality
\begin{eqnarray*}
\forall\;x'\in{X'},\,\, \forall\;E\in\vS, \qquad && \int_E\lg{x',f_1}\rg\,d\mu+\|x'\|\int_Er_1\,d\mu= \\ &&
\int_E\theta\lg{x',f_2}\rg\,d\mu+\|x'\|\int_Er_2|\theta|\,d\mu\,.
\end{eqnarray*}
But the sets $\left\{\dint_Ef_i\,d\mu: E\in\vS \right\},\;i=1,2$ are relatively weakly  compact and so there exists $0\neq{x'_0}\in{X'}$ vanishing on these sets.
 It follows that $r_1=r_2|\theta|\;\mu$-a.e. Appealing to \cite[Lemma]{mu4} we find that for each $x'\in{X'}$ one has $\lg{x',f_1}\rg=\theta\lg{x',f_2}\rg$ $\mu$-a.e.
  (the exceptional set depends on $x'$) i.e. $f_1$ is scalarly equivalent to $\theta\cdot f_2$. \\

One can easily check that also the reverse implication holds true: if there exists a measurable function $\theta$ such that  $f_1$ is scalarly equivalent to $\theta{f_2}$ and $r_1=r_2|\theta|$ $\mu$-a.e., then $M=\dint\theta\,dN$. \\

A similar calculation shows that $j\circ{M}$ can be represented as a $BDS$-integral with respect to $j\circ{N}$ if and only if the above $\theta$ is non-negative.\\
If we assume only that $r_1$ and $r_2$ are only positive and measurable, then we are in the local version of the RN-Theorem \ref{t4}.
\end{ex}
\begin{ex}\label{ex4}  \quad \rm   If we assume in Example \ref{ex3} that $X=Z'$ and $f_1,f_2$ are only Gelfand integrable then $\vG_1,\vG_2$ are weak$^*$
 multimeasures  (see \cite{mu4} for definitions). Applying now Remark \ref{r30}, we see that
$$
\forall\;z\in{Z}, \;\forall\;E\in\vS, \qquad s(z,M(E))=\int_E\theta^+\,ds(z,N)+\int_E\theta^-\,ds(z,-N)
$$
if and only there is $\theta$ such that $r_1=r_2 \cdot |\theta|\;\mu$-a.e. and $f_1$ is weak$^*$ scalarly equivalent to $\theta \cdot {f_2}$.\\
\end{ex}

{\small 
\noindent {\bf Conflict of interests:} The authors declare no conflict of interest.\\

{\bf \noindent Author Contributions:}
All  authors  have  contributed  equally  to  this  work  for  writing,  review  and  editing. All authors have read and agreed to the published version of the manuscript.
\\

{\bf \noindent Avaibility of matherials and data:} The authors confirm that the data supporting the results of this study are available within the article [and/or] its supplementary material in the bibliography.\\

{\bf \noindent Funding:}
This research has been accomplished within the UMI Group TAA “Approximation Theory and Applications”; it was  supported by
 Grant  “Analisi reale, teoria della misura ed approssimazione per la ricostruzione di im\-ma\-gini” (2020) of GNAMPA -- INDAM (Italy) and
by Ricerca di Base 2018 dell'Universit\`a degli Studi di Perugia - "Metodi di Teoria dell'Approssimazione, Analisi Reale, Analisi Nonlineare e loro Applicazioni";
 Ricerca di Base 2019 dell'Universit\`a degli Studi di Perugia - "Integrazione, Approssimazione, Analisi Nonlineare e loro Applicazioni"; "Metodi e processi innovativi per lo sviluppo di una banca di immagini mediche per fini diagnostici" funded by the Fondazione Cassa di Risparmio di Perugia (FCRP), 2018; "Metodiche di Imaging non invasivo mediante angiografia OCT sequenziale per lo studio delle Retinopatie degenerative dell'Anziano (M.I.R.A.)", funded by FCRP, 2019.\\

{\bf \noindent Acknowledgments\\}
This is a post-peer-review, pre-copyedit version of an article published in Journal of Convex Analysis. The final authenticated version is available
online at: https://www.heldermann.de/JCA/JCA29/JCA294/jca29062.htm
}


\Addresses

\end{document}